\documentclass[10pt,a4paper,oneside,openany]{book}

\usepackage{latexsym}
\usepackage{amsfonts}
\usepackage{amssymb}
\usepackage{amsmath}
\usepackage{amsthm}
\usepackage{setspace} 
\usepackage{color}
\usepackage{hyperref}
\usepackage{pdfpages}
\usepackage{paralist}  
\usepackage{etoolbox}
\patchcmd{\thebibliography}{\chapter*}{\section*}{}{}
\newcommand{\R}{\mathbb{R}}		% blackboard bold R
\newcommand{\beq}{\begin{equation}}		%starts equation with number
\newcommand{\eeq}{\end{equation}}			%ends equation with number
\newcommand{\beqq}{\begin{equation*}}	%starts equation without number
\newcommand{\eeqq}{\end{equation*}}		%ends equation without number

\newcommand{\id}{1\hspace{-0,9ex}1}

\renewcommand{\P}{\mathbb{P}}
\newcommand{\Id}{\text{Id}}

\newcommand{\B}{\mathfrak{B}} 
\newcommand{\F}{\mathcal{F}} 
\newcommand{\A}{\mathcal{A}} 

\newcommand{\Var}{\text{Var}}
\newcommand{\Cov}{\text{Cov}}
  
\newcommand{\x}{\text{\scalebox{0.62}{$\mathbb{X}$}}}
\newcommand{\X}{\mathbb{X}}

\allowdisplaybreaks

\newtheorem{theorem}{Theorem}[section]
\newtheorem{lemma}[theorem]{Lemma}
\newtheorem{definition}[theorem]{Definition}
\newtheorem{remark}[theorem]{Remark}

\newtheorem{notation}[theorem]{Notation}
\newtheorem{corollary}[theorem]{Corollary}
\newtheorem{assumption}[theorem]{Assumption}

\newtheorem{proposition}[theorem]{Proposition}

\newtheorem{stepp}{\noindent\bf{Step}}

{\begin{stepp}\rm}{\rm \end{stepp}}

\begin{document}
\pagestyle{headings} \thispagestyle{headings} \thispagestyle{empty}

\noindent\rule{15.812cm}{0.4pt}
\begin{center}
\textsc{Abstract Cauchy Problems in separable Banach Spaces driven by random Measures: Asymptotic Results in the finite extinction Case}
\end{center}
\begin{center}
	by
\end{center} 
\begin{center}
		\textsc{Alexander Nerlich\footnote{Affiliation: Ulm University}\footnote{Affiliation's address: 89081 Ulm, Helmholtzstr. 18, Germany}\footnote{Author's E-Mail: alexander.nerlich@uni-ulm.de}\footnote{Author's Phone: +49 176 6315 7836}\footnote{Author's ORCID: 0000-0001-7823-0648}}
\end{center}
\noindent\rule{15.812cm}{0.4pt}
\vspace{0.4cm}
\pagestyle{myheadings} 
\begin{center} 
	{\large ABSTRACT} 
\end{center} 
The aim of this paper is to prove the strong law of large numbers (SLLN) as well as the central limit theorem (CLT) for a class of vector-valued stochastic processes which arise as solutions of the stochastic evolution inclusion
\begin{align*}
 \eta(t,z) N_{\Theta}(dt \otimes z)\in dX(t)+\A X(t)dt,
\end{align*}
where $\A$ is a multi-valued operator and $N_{\Theta}$ is the counting measure induced by a point process $\Theta$.
The SLLN and the CLT will be proven not only for real-valued, but also for vector-valued functionals and the applicability of these results to the (weighted) $p$-Laplacian evolution equation (for "small" $p$) will be demonstrated.\\
The key assumption needed in this paper is that the nonlinear semigroup arising from the multi-valued operator $\A$ extincts in finite time.\\
\textbf{Mathematical Subject Classification (2010).} 47J35, 60H15, 35B40, 60F05, 60F15 \\ 
\textbf{Keywords.} Nonlinear (stochastic) evolution equation, Pure jump noise, Strong law of large numbers, Central limit theorem, Weighted $p$-Laplacian evolution equation, Anscombe's theorem

\section{Introduction}

Existence and uniqueness results for the stochastic evolution inclusion
\begin{align}
\tag{ACPRM}
\label{acprm}
\eta(t,z) N_{\Theta}(dt \otimes z)\in dX(t)+\A X(t)dt,
\end{align}
have been proven in \cite{ich2}. Moreover, a representation formula for the solutions was established there. In the current paper, we deduce intriguing asymptotic results with the aid of this representation formula.\\ 

Before stating our results as well as the required assumptions in more detail, let us give this formula. To this end, let $(V,||\cdot||_{V})$ denote a real, separable Banach space and let $\A:D(\A)\rightarrow 2^{V}$ be a densely defined, m-accretive operator. Then it is well known that the initial value problem
\begin{align} 
\label{acp}
0 \in u^{\prime}(t)+\A u(t),~\text{a.e. } t \in (0,\infty),~u(0)=v,
\end{align}
has for any $v \in V$ a uniquely determined mild solution, denoted by $T_{\A}(\cdot)v:[0,\infty)\rightarrow V$, see  \cite[Prop. 3.7]{BenilanBook}.\\
Now, introduce a complete probability space $(\Omega,\F,\P)$ and let $(\beta_{m})_{m \in \mathbb{N}}$ and $(\eta_{m})_{m \in \mathbb{N}}$ be $(0,\infty)$-valued and $V$-valued sequences of random variables, respectively. In addition, let $x$ be a $V$-valued random variable, introduce $\alpha_{m}:=\sum \limits_{k=1}\limits^{m}\beta_{m}$, $\alpha_{0}:=0$, $\x_{x,0}:=x$ and $\x_{x,m}:=T_{\A}(\beta_{m})\x_{x,m-1}+\eta_{m}$ for all $m \in \mathbb{N}$. Then the stochastic process $\X_{x}:[0,\infty)\times \Omega \rightarrow V$ defined by
\begin{align}
\label{intro_repr}
\X_{x}(t):= \sum \limits_{m=0}\limits^{\infty}T_{\A}((t-\alpha_{m})_{+})(\x_{x,m})\id_{[\alpha_{m},\alpha_{m+1})}(t),
\end{align}
is for some drift $\eta$ and some random measure $N_{\Theta}$ the uniquely determined mild solution of (\ref{acprm}), starting at $x$, if  $\A$ fulfills certain regularity assumptions and $(\beta_{m})_{m \in \mathbb{N}}$ is i.i.d, see \cite[Theorem 3.13 and Remark 3.14]{ich2}.\\

Even though this representation formula does not make it possible to explicitly calculate the solution of (\ref{acprm}), it still gives a direct link between the solution of the deterministic Cauchy problem (\ref{acp}) and (\ref{acprm}). Consequently, it raises the question, how the asymptotic properties of $T_{\A}$ and $\X_{x}$ are related. The probably strongest asymptotic property $T_{\A}$ can have, is that $T_{\A}(\cdot)v$ extincts in finite time, which is in our case managed by assuming that: There are constants $\kappa \in (0,\infty)$ and $\rho \in (0,1)$ such that 
\begin{align}
\label{intro_ex}
||T_{\A}(t)v||^{\rho}_{V_{1}} \leq (-\kappa t+||v||_{V_{1}}^{\rho})_{+}
\end{align}
for all $t \geq 0$ and $v \in V_{1}$, where $(V_{1},||\cdot||_{V_{1}})\subseteq V$ is another separable Banach space, invariant w.r.t. $T_{\A}$ and continuously injected into $V$. The reason why we introduce $V_{1}$ is to make the results more applicable, since it is quite common that it is possible to prove existence and uniqueness of mild solutions of (\ref{acp}) for all $v \in V$, but that the finite extinction property (\ref{intro_ex}) only holds on a subspaces.\\
The most important stochastic assumptions needed to achieve this, are that $(\beta_{m})_{m \in \mathbb{N}}$ and $(\eta_{m})_{m \in \mathbb{N}}$ are both i.i.d. sequences, which are independent of each other, independent of the initial $x$ and that $\beta_{m}$ is in some sense (to be made precise later) "larger" than $\eta_{m}$.\\ 
It will then be possible to show that, for a class of functionals $\Xi:V \rightarrow W$, where $(W,||\cdot||_{W})$ is another separable Banach space, we have
\begin{align}
\tag{SLLN}
\label{slln}
\lim \limits_{t \rightarrow \infty}\frac{1}{t} \int \limits_{0}\limits^{t} \Xi (\X_{x}(\tau))d\tau=\nu_{\Xi},
\end{align}
with probability one, where $\nu_{\Xi}\in W$ will be made precise later; and that if $(W,||\cdot||_{W})$ is in addition a type $2$ Banach space, we have
\begin{align}
\tag{CLT}
\label{clt}
\lim \limits_{t \rightarrow \infty} \frac{1}{\sqrt{t}}\left(\int \limits_{0}\limits^{t} \Xi (\X_{x}(\tau))d\tau-t\nu_{\Xi}\right)=Z,
\end{align}
in distribution, where $Z:\Omega \rightarrow W$ is a centered, Gaussian $W$-valued random variables, whose covariance will be made precise later.\\
Particularly, the class of functionals is sufficiently large, such that $\Xi(\X_{x}(t))$ in (\ref{slln}) and (\ref{clt}) can be replaced by $\X_{x}(t)$. Moreover, $\Xi$ depends on another separable Banach space $(V_{2},||\cdot||_{V_{2}})\subseteq (V,||\cdot||_{V})$, with continuous injection and invariant w.r.t.  $T_{\A}$. This makes it possible to replace $\Xi(\X_{x}(t))$ in (\ref{slln}) and (\ref{clt}) by $||\X_{x}(t)||_{V_{2}}$.\\
All of these results are proven solely with the aid of the representation formula (\ref{intro_repr}); particularly, no precise notion of a solution of (\ref{acprm}) is required.\\

Moreover, our theoretical results will be applied to the weighted $p$-Laplacian evolution equation on an $L^{1}$-space, where $p \in I$ and $I \subseteq (1,2)$ is an interval to be specified later. (The usual $p$-Laplacian evolution equation is a special case of this equation, with the weight function being equal to one.) We will see that in this case all $L^{q}$-norms, where $q\in [1,\infty)$, are a valid choice for $||\cdot||_{V_{2}}$ and that (\ref{slln}) as well as (\ref{clt}) also hold for $\X_{x}$ itself.\\

The basic technique to prove these results is to introduce a certain sequence of stopping times $(\tau_{m})_{m\in \mathbb{N}}$, such that $\int \limits_{0}\limits^{\tau_{m}}\Xi (\X_{x}(\tau))d\tau$ can be decomposed into an i.i.d. sum; and then to use approximation techniques to replace $\tau_{m}$ by $t$.\\

Results like (\ref{slln}) and (\ref{clt}) are relatively rare in the field of nonlinear SPDEs; in particular, it is rare that it is possible to prove them for vector-valued functionals and not just for real-valued. Moreover, the only structural assumption needed regarding $V$ is that it is separable. Even though we also consider a triplet of 3 Banach spaces, $V_{2},V_{1},V$, we do not assume that these Banach spaces form a Gelfand triplet, but simply that all of them are separable and that the injections are continuous.\\

There are besides the weighted $p$-Laplacian example we consider, many other nonlinear semigroups which extinct in finite time. For another concrete example, see \cite[Chapter 4]{acmbook} and for a general survey on the finite extinction property, containing many examples, including the (unweighted) $p$-Laplacian case, see \cite{extinct}.\\ 
Proving asymptotic results of $\X_{x}$ under the assumption that $T_{\A}$ fulfills other decay estimates than (\ref{intro_ex}) is the subject of current research.\\

Before embarking on the endeavor ahead of us, let us outline this paper's structure: All notations and basic results used throughout this paper are stated in Section \ref{basicdef}. Section \ref{section_asymptotics} is this paper's core; a precise statement of all assumptions needed and proofs of the general results mentioned in the introduction are given there. Finally, Section \ref{plaplace} deals with the application of these results to the weighted $p$-Laplacian evolution equation.\\
Section \ref{section_asymptotics} also contains a type $2$ Banach space version of Anscombe's CLT, which we did not find in the literature and might be of independent interest to some readers. It can be found in Theorem \ref{theorem_anscombeclt} and is written as self-contained as possible.\\

\section{Notation and preliminary Results}
\label{basicdef}
Firstly, let us state some functional analytic preliminaries: Whenever $(U,||\cdot||_{U})$ is a Banach space, $U^{\prime}$ denotes its dual and $\langle \cdot,\cdot\rangle_{U}$ the duality between $U$ and $U^{\prime}$.\\ 
Throughout this section, $(U,||\cdot||_{U})$ denotes a separable real Banach space.\\ 
If $(K,\Sigma,\mu)$ is a $\sigma$-finite measure space, then $L^{q}(K,\Sigma,\mu;U)$ denotes, for any $q\in[1,\infty)$, the set of all (equivalence classes of) functions  $f:K\rightarrow U$ which are $\Sigma-\B(U)$-measurable and fulfill
\begin{align*}
\int  \limits_{K} ||f||_{U}^{q}d\mu < \infty,
\end{align*}
where $\B(T)$ always denotes the Borel $\sigma$-algebra of a topological space $(T,\mathcal{T})$. For any $f \in L^{q}(K,\Sigma,\mu;U)$, the integral $\int \limits_{K}fd\mu$ is understood as a Bochner integral; for an introduction to Bochner integrability, see \cite[Section 2.1]{SIBS}.\\ 
Now we also need some results regarding nonlinear semigroups. The reader is referred to \cite{BenilanBook} for a comprehensive introduction to this topic. Moreover, \cite{acmbook} deals with existence, uniqueness and asymptotic results for many initial value problems; and this book's appendix contains a more concise introduction to nonlinear semigroups.\\
Now, let $\mathcal{A}:U\rightarrow 2^{U}$ be a multi-valued operator, then we introduce $D(\mathcal{A}):=\{u\in U: \mathcal{A}u\neq \emptyset\}$ and we call this operator single-valued if $\mathcal{A}u$ contains precisely one element for all $u\in D(\mathcal{A})$. Moreover, instead of $\mathcal{A}:U\rightarrow 2^{U}$ we may write $\mathcal{A}:D(\mathcal{A})\rightarrow 2^{U}$. In addition, $G(\A):=\{(v,\hat{v}):~v \in D(\A),~\hat{v}\in \A v\}$ is the graph of $\A$. By identifying an operator with its graph, we may simply write $(v,\hat{v})\in \A$ instead of $v \in D(\A)$ an $\hat{v}\in \A v$.\\
Moreover, $\A$ is called accretive, if $||u_{1}-u_{2}||_{U}\leq ||u_{1}-u_{2}+\alpha(\hat{u}_{1}-\hat{u}_{2})||_{U}$ for all $\alpha>0$, $(u_{1},\hat{u}_{1})\in \A$ and $(u_{2},\hat{u}_{2})\in \A$; m-accretive, if it is accretive and $R(Id+\alpha\mathcal{A})=U$ for all $\alpha>0$; and densely defined, if $\overline{D(\mathcal{A})}=U$.\\
 Using these simple definitions enables us to invoke the following well-known result:

\begin{remark}\label{remark_msex} Let $\mathcal{A}:D(\mathcal{A})\rightarrow 2^{U}$ be m-accretive and densely defined. Then, for any $u\in U$, the initial initial value problem
	\begin{align}
	\label{remark_mseq}
	0 \in v^{\prime}(t)+\mathcal{A}v(t),~\text{for a.e. }t\in (0,\infty),~v(0)=u,
	\end{align}
	has precisely one mild solution. The reader is referred to \cite[Prop. 3.7]{BenilanBook} for a proof and to \cite[Definition 1.3]{BenilanBook} for the definition of mild solution.\\
	For a given m-accretive and densely defined operator $\mathcal{A}:D(\mathcal{A})\rightarrow 2^{U}$, we denote for each $u \in U$ by $T_{\mathcal{A}}(\cdot)u:[0,\infty)\rightarrow U$ the uniquely determined mild solution of (\ref{remark_mseq}). It is well known (see \cite[Theorem 3.10]{BenilanBook} and \cite[Theorem 1.10]{BenilanBook}) that the family of mappings $(T_{\A}(t))_{t \geq 0}$ forms a jointly continuous, contractive semigroup, i.e. it fulfills
	\begin{enumerate}
		\item semigroup property: $T_{\A}(0)u=u$ and $T_{\A}(t+h)u=T_{\A}(t)T_{\A}(h)u$ for all $t,h \in [0,\infty)$ and $u \in U$
		\item joint-continuity: $[0,\infty)\times U \ni (t,u) \mapsto T_{\A}(t)u$ is a continuous map, and
		\item contractivity: $||T_{\A}(t)u_{1}-T_{\A}(t)u_{2}||_{U}\leq ||u_{1}-u_{2}||_{U}$ for all $t \in [0,\infty)$. $u_{1},~u_{2}\in U$.
	\end{enumerate}
	In the sequel, we refer to the family of mappings $(T_{\A}(t))_{t\geq 0}$ as the semigroup associated to $\A$.
\end{remark} 

\begin{remark}\label{remark_measuc0sg}  Let $\mathcal{A}:D(\mathcal{A})\rightarrow 2^{U}$ be m-accretive and densely defined. As $[0,\infty) \times U \ni (t,u)\mapsto T_{\A}(t)u$ is continuous it is a fortiori $\mathfrak{B}([0,\infty)\times U)$-$\mathfrak{B}(U)$-measurable. Moreover, by separability we have $\mathfrak{B}([0,\infty)\times U)= \mathfrak{B}([0,\infty))\otimes \mathfrak{B}(U)$, see \cite[page 244]{Billingsley}; which gives that this map is $\mathfrak{B}([0,\infty))\otimes \mathfrak{B}( U)$-$\mathfrak{B}(U)$-measurable.
\end{remark}

\begin{definition} Let $\mathcal{A}:D(\mathcal{A})\rightarrow 2^{U}$ be m-accretive and densely defined. Moreover, let $\tilde{U}\subseteq U$. Then we say that $\tilde{U}$ is an invariant set w.r.t. $T_{\A}$, if $T_{\A}(t)\tilde{u} \in \tilde{U}$ for all $t \in [0,\infty)$ and $\tilde{u}\in\tilde{U}$.
\end{definition}

Now let us proceed with the stochastic preliminaries, which are mainly concerned with vector-valued random variables, i.e. random variables taking values in a (separable) Banach space. The reader is referred to \cite{PIBS} for a comrepehnsive introduction to this topic.\\
Throughout everything which follows $(\Omega,\F,\P)$ denotes a complete probability space. Moreover, we introduce the short cut notation $L^{q}(\Omega,\F,\P;U):=L^{q}(\Omega;U)$ for all $q \in [1,\infty)$. In addition, if $U=\R$ we may simply write $L^{q}(\Omega)$. Furthermore, $\mathcal{M}(\Omega;U)$ denotes the space of all mappings $Y:\Omega \rightarrow U$ which are $\F$-$\B(U)$-measurable. We may also refer to the elements of $\mathcal{M}(\Omega;U)$ as $U$-valued random variables. Moreover, if $Y_{i}$ is a $U_{i}$-valued random variable for each $i \in I$, where $I$ is an arbitrary index set and the $U_{i}$'s are separable Banach spaces, then $\sigma(Y_{j};j \in I)\subseteq \F$ is the smallest  $\sigma$-Algebra, such that each $Y_{i}$ is $\sigma(Y_{j};j \in I)-\B(U_{i})$-measurable. In addition, $\sigma_{0}(Y_{j};j \in I)$ denotes its completion, i.e. 
\begin{align*}
\sigma_{0}(Y_{j};j \in I):=\{A \in \F:~ \exists B \in \sigma(Y_{j};j \in I)\text{, such that } \P(A \Delta B)=0 \},
\end{align*}
where $\Delta$ denotes the symmetric difference. It is easily verified that the right-hand-side of the previous equation is indeed a $\sigma$-Algebra and the smallest one containing all $\P$-null-sets as well as all elements of $\sigma(Y_{j};j \in I)$. Moreover, it is well known that an $Y \in \mathcal{M}(\Omega;U)$ is independent of a $\sigma$-algebra, if and only if it is independent of the $\sigma$-algebra's completion.\\ 
Now let us recall the notations regarding Gaussian (vector-valued) random variables and convergence in distribution, needed in the sequel:

\begin{remark} The separable Banach space $(U,||\cdot||_{U})$ is said to be of type $2$, if: There is a constant $C>0$ such that for all $n \in \mathbb{N}$, $X_{1},...,X_{n} \in L^{2}(\Omega;U)$ which are centered and independent, we have
	\begin{align*}
	\mathbb{E} \left|\left|\sum \limits_{k=1}\limits^{n}X_{k}\right|\right|_{U}^{2} \leq C \sum \limits_{k=1}\limits^{n}\mathbb{E}||X_{k}||_{U}^{2}.
	\end{align*}
	The main feature of such Banach spaces is that these are precisely the Banach spaces where every centered, square integrable i.i.d. sequence still fulfills the CLT, see \cite[Theorem 10.5]{PIBS}.\\
	Now let $Y\in \mathcal{M}(\Omega;U)$. Then $Y$ is called Gaussian, if $\langle Y,\psi\rangle_{U}$ is Gaussian for all $\psi \in U^{\prime}$. (Note that by this definition constant random variables are Gaussian as well.) In addition, for a (not necessarily Gaussian) random variable $Y\in L^{2}(\Omega;U)$, we call the mapping $\Cov_{U}(Y):U^{\prime}\times U^{\prime}\rightarrow \mathbb{R}$, where
	\begin{align*}
	\Cov_{U}(Y)(\psi_{1},\psi_{2}):=\mathbb{E}(\langle Y-\mathbb{E}Y,\psi_{1}\rangle_{U}\langle Y-\mathbb{E}Y,\psi_{2}\rangle_{U}),~\forall \psi_{1},~\psi_{2}\in U^{\prime},
	\end{align*} 
	the covariance of $Y$. It is plain to verify that the right-hand-side expectation in the preceding equation indeed exists.\\ 
	Moreover, if $Y\in \mathcal{M}(\Omega;U)$ is Gaussian, then particularly $Y \in L^{2}(\Omega;U)$, see \cite[p. 5]{PR}. In addition, analogously to the real-valued case, the distribution of $Y$ is still uniquely determined by $\mathbb{E}Y$ and $\Cov_{U}(Y)$, see \cite[p. 5]{PR}.\\
	In the sequel, it will be written $Y \sim N_{U}(\mu,Q)$ whenever $Y\in \mathcal{M}(\Omega;U)$ is Gaussian, with mean $\mu$ and covariance $Q$. Of course, if $U=\mathbb{R}$ this is abbreviated by $N(\mu,\sigma^{2})$, where $\sigma^{2}:=Q(\Id,\Id)$ is the variance of $Y$. In addition, for any $Y \in \mathcal{M}(\Omega;U)$, $\P_{Y}:\B(U)\rightarrow [0,1]$ denotes its law.\\
	Last but not least, let us remark, that as usually we say that $\lim \limits_{m\rightarrow \infty}Y_{m}=Y$ in distribution, where $Y_{m},~Y\in \mathcal{M}(\Omega;U)$, if $\lim \limits_{m \rightarrow \infty}\mathbb{E}f(Y_{m})=\mathbb{E}f(Y)$, for all $f:U\rightarrow \mathbb{R}$ which are continuous and bounded.
\end{remark}

Finally, let us spend some words on the stochastic process which is the central object of this paper:

\begin{definition} Let $(\beta_{m})_{m \in \mathbb{N}}$, where $\beta_{m}:\Omega \rightarrow (0,\infty)$, be a sequence of real-valued random variables. Moreover, let $(\eta_{m})_{m \in \mathbb{N}}\subseteq\mathcal{M}(\Omega;U)$, introduce $\alpha_{m}:=\sum \limits_{k=1}\limits^{m}\beta_{k}$ for all $m \in \mathbb{N}$ and set $\alpha_{0}:=0$. Finally, let $x \in \mathcal{M}(\Omega;U)$ and let $\A:D(\A)\rightarrow 2^{U}$ be m-accretive and densely defined. Then the sequence $(\x_{x,m})_{m\in \mathbb{N}_{0}}$ defined by $\x_{x,0}:=x$ and
\begin{align*}
\x_{x,m}:=T_{\A}(\alpha_{m}-\alpha_{m-1})\x_{x,m-1}+\eta_{m}=T_{\A}(\beta_{m})\x_{x,m-1}+\eta_{m},~\forall m \in \mathbb{N},
\end{align*}
is called the sequence generated by $((\beta_{m})_{m \in \mathbb{N}},(\eta_{m})_{m \in \mathbb{N}},x,\A)$ in $U$. Moreover, the stochastic process $\X_{x}:[0,\infty)\times \Omega \rightarrow U$ defined by
\begin{align*}
\X_{x}(t):= \sum \limits_{m=0}\limits^{\infty}T_{\A}((t-\alpha_{m})_{+})(\x_{x,m})\id_{[\alpha_{m},\alpha_{m+1})}(t),~\forall t\geq 0.
\end{align*}
is called the process generated by $((\beta_{m})_{m \in \mathbb{N}},(\eta_{m})_{m \in \mathbb{N}},x,\A)$ in $U$.
\end{definition}
\interfootnotelinepenalty=10000
\begin{remark}\label{remark_Xmeascad} Let $(\x_{x,m})_{m\in \mathbb{N}_{0}}$ and $\X_{x}:[0,\infty)\times \Omega \rightarrow U$ be the sequence and the process generated by some $((\beta_{m})_{m \in \mathbb{N}},(\eta_{m})_{m \in \mathbb{N}},x,\A)$ in $U$. Then it follows easily from Remark \ref{remark_measuc0sg} that each $\x_{x,m}$ and each $\X_{x}(t)$ is $\F$-$\B(U)$-measurable and that $\X_{x}$ has almost surely c\`{a}dl\`{a}g paths. Consequently, this process is $\B([0,\infty))\otimes \F$-$\B(U)$-measurable\footnote{Actually, this only implies that there is a process indistinguishable of $\X_{x}$ which is $\B([0,\infty))\otimes \F$-$\B(U)$-measurable. However, we follow the common mathematical convention of identifying indistinguishable processes with each other.}, see \cite[Prop. 2.2.3]{SIBS}.
\end{remark}

\begin{remark} Let $(\x_{x,m})_{m\in \mathbb{N}_{0}}$ and $\X_{x}:[0,\infty)\times \Omega \rightarrow U$ be the sequence and the process generated by some $((\beta_{m})_{m \in \mathbb{N}},(\eta_{m})_{m \in \mathbb{N}},x,\A)$ in $U$. Moreover, assume $\eta_{m}=0$ for all $m \in \mathbb{N}$ almost surely. Then we have $\x_{x,m}=T_{\A}(\beta_{m})\cdot...\cdot T_{\A}(\beta_{1})x$ for all $m \in \mathbb{N}$ a.s., and thus, thanks to the semigroup property, we get $\x_{x,m}=T_{\A}(\alpha_{m})x$, for all $m \in \mathbb{N}_{0}$, almost surely. Consequently, employing the semigroup property once more, yields $\X_{x}(t)=T_{\A}(t)x$ for all $t \geq 0$, with probability one. This demonstrates that even for the most simple noise, i.e. $\eta_{m}=0$, one needs some assumptions regarding the asymptotic behavior of $T_{\A}$, to be able to prove asymptotic results like (\ref{slln}) or (\ref{clt}).
\end{remark}

\begin{remark}\label{lemma_meas1} Let $(\hat{U},||\cdot||_{\hat{U}})\subseteq (U,||\cdot||_{U})$ be another separable Banach space and assume that the injection $\hat{U}\hookrightarrow U$ is continuous. Then Lusin-Souslin's Theorem (see \cite[Theorem 15.1]{kechris}) yields \linebreak$f(B)\in \B(U)$ for all $B \in \B(\hat{U})$ and $f:\hat{U}\rightarrow U$ which are continuous and injective. Consequently, we get $\B(\hat{U})\subseteq \B(U)$. Particularly, for $|\cdot|_{\hat{U}}:U\rightarrow [0,\infty)$, with $|u|_{\hat{U}}:=||u||_{\hat{U}}$ for all $u \in \hat{U}$ and $|u|_{\hat{U}}:=0$ for all $u \in U \setminus \hat{U}$, we have that $|\cdot|_{\hat{U}}$ is $\B(U)$-$\B([0,\infty))$-measurable.\\
Hence, if $y:\Omega \rightarrow U$ is $\F$-$\B(U)$-measurable, with $\P(y\in \hat{U})=1$, then $||y||_{\hat{U}}$ is $\F$-$\B([0,\infty))$-measurable.
\end{remark}

\section{Asymptotic Results for abstract Cauchy Problems driven by random Measures}
\label{section_asymptotics} 
The purpose of this section is to prove the introductory mentioned results (\ref{slln}) and (\ref{clt}) . At first we will state the needed assumptions, as well as some additional notations. As this section is quite long, a detailed outline is given after all the assumptions and notations have been stated, see Remark \ref{remark_outline}. There the basic techniques which are employed to prove (\ref{slln})  and (\ref{clt})  are also described.\\

Throughout this section, $(V,||\cdot||_{V})$ denotes a real, separable Banach space; and $\A:D(\A)\rightarrow 2^{V}$ is a densely defined, m-accretive operator. In addition, $(T_{\A}(t))_{t\geq 0}$ denotes the semigroup associated to $\A$. Finally, the following functional analytic assumption is drawn: 

\begin{assumption}\label{assumption_fe} There are separable Banach spaces $(V_{1},||\cdot||_{V_{1}})$ and $(V_{2},||\cdot||_{V_{2}})$, with $V_{i}\subseteq V$, such that the injections $V_{i} \hookrightarrow V$ are continuous for $i=1,2$. In addition, the following assertions hold.
\begin{enumerate}
	\item $V_{i}$ is invariant w.r.t. $T_{\A}$ for $i=1,2$.
	\item There are constants $\kappa \in (0,\infty)$ and $\rho \in (0,1)$ such that $||T_{\A}(t)v||^{\rho}_{V_{1}} \leq (-\kappa t+||v||_{V_{1}}^{\rho})_{+}$ for all $t \geq 0$ and $v \in V_{1}$, where $(\cdot)_{+}:=\max(\cdot,0)$.
	\item $||T_{\A}(t)v||_{V_{2}}\leq ||v||_{V_{2}}$ for all $v \in V_{2}$.
\end{enumerate}
\end{assumption}
Throughout this entire section, Assumption \ref{assumption_fe} is assumed to hold and $(V_{1},||\cdot||_{V_{1}})$, $(V_{2},||\cdot||_{V_{2}})$ as well as $\kappa\in (0,\infty)$, $\rho\in (0,1)$ are as in this assumption.\\

Now let us proceed with the stochastic assumptions and notations. Throughout this section, let $(\eta_{m})_{m \in \mathbb{N}}$ and $(\beta_{m})_{m \in \mathbb{N}}$ denote i.i.d. sequences, where $\eta_{m}:\Omega \rightarrow V$ and $\beta_{m}:\Omega \rightarrow (0,\infty)$ are $\F$-$\B(V)$-measurable and $\F$-$\B((0,\infty))$-measurable, respectively.  In addition, assume that $(\eta_{m})_{m \in \mathbb{N}}$ and $(\beta_{m})_{m \in \mathbb{N}}$ are independent of each other. Finally, introduce $(\alpha_{m})_{m \in \mathbb{N}_{0}}$, where $\alpha_{m}:\Omega \rightarrow [0,\infty)$, by $\alpha_{0}:=0$ and
\begin{align*}
\alpha_{m}:=\sum\limits_{k =1}\limits^{m}\beta_{k},~\forall m \in \mathbb{N}.
\end{align*}
The final assumption needed, reads as follows:

\begin{assumption}\label{assumption_mb} Throughout this section, the following assertions hold for all $m \in \mathbb{N}$.
\begin{enumerate}
	\item $\eta_{m}\in V_{i}$ for $i=1,2$ with probability one.
	\item $\mathbb{E}||\eta_{m}||_{V_{2}}^{4}<\infty$, and there is a $\hat{\varepsilon}>0$, such that $\mathbb{E}||\eta_{m}||_{V_{1}}^{\rho(11+\hat{\varepsilon})}<\infty$ and $\mathbb{E}\beta_{m}^{11+\hat{\varepsilon}}<\infty$. 
	\item $-\kappa\mathbb{E}\beta_{m}+\mathbb{E}||\eta_{m}||_{V_{1}}^{\rho}<0$.
\end{enumerate}	
\end{assumption}

Throughout this section $\hat{\varepsilon}>0$ is as in the preceding assumption. Moreover, for any $x \in \mathcal{M}(\Omega;V)$, $(\x_{x,m})_{m\in \mathbb{N}_{0}}$ and $\X_{x}:[0,\infty)\times \Omega \rightarrow V$ denote the sequence and the process generated by \linebreak$((\beta_{m})_{m \in \mathbb{N}},(\eta_{m})_{m \in \mathbb{N}},x,\A)$ in $V$. 

\begin{notation}\label{notation_xi} We write $(\Xi,(W,||\cdot||_{W})) \in \text{SL}_{V_{2}}(V)$, whenever the following assertions hold.
	\begin{enumerate}
		\item $(W,||\cdot||_{W})$ is a separable Banach space.
		\item $\Xi:V\rightarrow W$ is $\B(V)-\B(W)$-measurable.
		\item $\Xi$ is sub-linear in the following sense: There are constants $c_{1},~c_{2}\in [0,\infty)$ such that\linebreak $||\Xi(v)||_{W}\leq c_{1}||v||_{V_{2}}+c_{2}$, for all $v \in V_{2}$. 
	\end{enumerate}
\end{notation}

\begin{definition}\label{definition_basic} A mapping $x:\Omega \rightarrow V$ is called an independent initial value leading to extinction, if the following assertions hold.
\begin{enumerate}
	\item $x \in \mathcal{M}(\Omega;V)$.
	\item $x \in V_{i}$ for $i=1,2$ with probability one. 
	\item $\mathbb{E}||x||^{2\rho}_{V_{1}} < \infty$. 
	\item $x$ is jointly independent of $(\beta_{m})_{m \in \mathbb{N}}$ and $(\eta_{m})_{m \in \mathbb{N}}$.
\end{enumerate}
Moreover, if $x:\Omega \rightarrow V$ is an independent initial leading to extinction, we denote by $(e_{x}(n))_{n \in \mathbb{N}}$, where $e_{x}(n):\Omega \rightarrow \mathbb{N}\cup \{\infty\}$, the sequence of extinction times, defined by
\begin{enumerate}\setcounter{enumi}{3}
	\item $e_{x}(1):= \min (m \in \mathbb{N}:~T_{\A}(\beta_{m})\x_{x,m-1}=0)$ and
	\item $e_{x}(n):= \min \big(m \in \mathbb{N}:~T_{\A}(\beta_{m})\x_{x,m-1}=0,~m>e_{x}(n-1)\big)$ for all $n \in \mathbb{N}\setminus\{1\}$.
\end{enumerate}
Finally, introduce the filtrations $(\F^{x}_{j})_{j \in \mathbb{N}}$ and $(\tilde{\F}^{x}_{m})_{m \in \mathbb{N}_{0}}$, by
\begin{enumerate}\setcounter{enumi}{5}
	\item $\F^{x}_{1}:=\sigma_{0}(x,\beta_{1})$, $\tilde{\F}^{x}_{0}:=\sigma_{0}(x)$ and
	\item $\F^{x}_{j}:=\sigma_{0}(x,\beta_{1},...,\beta_{j},\eta_{1},...,\eta_{j-1})$ for all $j \in \mathbb{N}\setminus\{1\}$ and $\tilde{\F}^{x}_{m}:=\sigma_{0}(x,\beta_{1},...,\beta_{m},\eta_{1},...,\eta_{m})$ for all $m \in \mathbb{N}$.
\end{enumerate}
\end{definition}

\begin{remark}\label{remark_outline} Let $x \in \mathcal{M}(\Omega;V)$ be an independent initial leading to extinction and $\Xi \in SL_{V_{2}}(V)$. The centerpiece of the proof of the SLLN as well as the CLT, which are both proven in Theorem \ref{theorem_sllncltmain}, is the fact that the sequence $\left(\int \limits_{\alpha_{e_{x}(n)}}\limits^{\alpha_{e_{x}(n+1)}}\Xi(\X_{x}(\tau))d\tau\right)_{n\in \mathbb{N}}$ is i.i.d., square integrable and for each $n \in \mathbb{N}$ in distribution equal to $\int \limits_{0}\limits^{\alpha_{e_{\overline{x}}(1)}}\Xi(\X_{\overline{x}}(\tau))d\tau$, where $\overline{x}\in \mathcal{M}(\Omega;V)$ is specified in Remark \ref{remark_xbar}.\\
Before one can prove these results, one of course needs that $\P(e_{x}(n)<\infty,~\forall n \in \mathbb{N})=1$ and that the occurring integrals exist and are well-defined, which is subject to Proposition \ref{prop_extinction} and Lemma \ref{lemma_Xbasicprop}. The stated i.i.d. and square integrability assertions are then proven in Proposition \ref{prop_iid} and Lemma \ref{lemma_squareintegrability}.\\
Even though $\left(\int \limits_{\alpha_{e_{x}(n)}}\limits^{\alpha_{e_{x}(n+1)}}\Xi(\X_{x}(\tau))d\tau\right)_{n\in \mathbb{N}}$ is i.i.d., it remains so far open how one gets from there to Theorem \ref{theorem_sllncltmain}. A similar obstacle occurs for discrete time Markov chains possessing an atom; and the technique we employ to overcome it is somehow similar to the one used in \cite[Theorems 17.2.1 and 17.2.2]{tweety}. It is just "somehow" similar, since we are not in discrete time, consider vector-valued instead of real-valued functionals and last but not least $T_{\A}(\beta_{m})\x_{x,m-1}=0$, means $\x_{x,m}=\eta_{m}$, i.e. we do not stop the sequence $(\x_{x,m})_{m \in \mathbb{N}}$ at deterministic states, but at a "random state"; moreover, note that even though $(\x_{x,m})_{m \in \mathbb{N}}$ is a Markov chain, $\X_{x}$ is not necessarily\footnote{The author conjectures that the only nontrivial distribution of $\beta_{m}$ which turns $\X_x$ in a Markov process is the exponential one.} a Markov process. \\
Moreover, Corollary \ref{theorem_corvectorvalued} is a useful applications of Theorem \ref{theorem_sllncltmain} for special choices of $(\Xi,(W,||\cdot||_{W}))$. In addition, Theorem \ref{theorem_anscombeclt} is a vector-valued version of Anscombe's CLT.\\ 
The remaining results, which have not been mentioned explicitly in this remark, solely serve to keep the exposition more clean and the proofs more accessible, but are not of independent interest out of this section.
\end{remark} 

\begin{lemma}\label{lemma_xbasicprop} Let $x:\Omega \rightarrow V$ be an independent initial leading to extinction. Then all of the following assertions hold.
\begin{enumerate}
	\item $\x_{x,m}$ is $\tilde{\F}^{x}_{m}$-$\B(V)$-measurable for all $m \in \mathbb{N}_{0}$.
	\item $e_{x}(n)+1\leq e_{x}(n+1)$ and $e_{x}(n)\geq n$ for all $n \in \mathbb{N}$.
	\item $\{e_{x}(n)=j\}\in \F^{x}_{j}$ for all $n,j\in\mathbb{N}$.
\end{enumerate}
\end{lemma}
\begin{proof} Let us start by proving i) inductively. We have $\x_{x,0}=x$, which is obviously $\sigma_{0}(x)$-$\B(V)$-measurable. Now assume that i) holds for an $m \in \mathbb{N}_{0}$ and note that $\x_{x,m+1}=T_{\A}(\beta_{m+1})\x_{x,m}+\eta_{m+1}$. As $\x_{x,m}$ is by the induction hypothesis a fortiori $\tilde{\F}^{x}_{m+1}$-$\B(V)$-measurable and since $\beta_{m+1}$ is obviously $\tilde{\F}^{x}_{m+1}$-$\B([0,\infty))$-measurable, Remark \ref{remark_measuc0sg} yields that $T_{\A}(\beta_{m+1})\x_{x,m}$ is $\tilde{\F}^{x}_{m+1}$-$\B(V)$-measurable. As $\eta_{m+1}$ has this property as well, i) follows.\\
Now note that it is plain that $e_{x}(n)+1\leq e_{x}(n+1)$, which gives $e_{x}(n)\geq n$, since $e_{x}(1)\geq 1$, by definition. Consequently, ii) holds as well.\\
Proof of iii). This statement is proven inductively w.r.t. $n \in \mathbb{N}$. We have for any $j \in \mathbb{N}$ that
\begin{align*}
\{e_{x}(1)\leq j\}=\{\exists k \in \{1,...,j\}:~T(\beta_{k})\x_{x,k-1}=0\}=\bigcup \limits_{k=1}\limits^{j}\{T(\beta_{k})\x_{x,k-1}=0\} \in \F^{x}_{j},
\end{align*}
by Remark \ref{remark_measuc0sg} and i). Consequently, as $\{e_{x}(1)= j\}=\{e_{x}(1)\leq j\}\setminus \{e_{x}(1)\leq j-1\}$ and $\F_{j-1}^{x}\subseteq \F_{j}^{x}$,  iii) holds if $n=1$.\\
Now assume that iii) holds for an $n \in \mathbb{N}$. If $j<n+1$, we have $\{e_{x}(n+1)\leq j\}=\emptyset$, by ii). So let $j\geq n+1$. Note that on $\{e_{x}(n+1)\leq j\}$, we have $n\leq e_{x}(n)<j$, by ii).\\
Consequently, we have
\begin{align*}
\{e_{x}(n+1)\leq j\}=\bigcup\limits_{i=n}\limits^{j-1}\{\exists~k\in\{i+1,...,j\}:~T(\beta_{k})\x_{x,k-1}=0,~e_{x}(n)=i\}.
\end{align*}
Moreover, the induction hypothesis yields $\{e_{x}(n)=i\}\in \F_{i}^{x}\subseteq \F_{j}^{x}$, for all $i=n,...,j-1$ and combining Remark \ref{remark_measuc0sg} and i) gives $\{T(\beta_{k})\x_{x,k-1}=0\}\in \F_{k}^{x}\subseteq \F_{j}^{x}$ for all $k=n+1,...,j$.\\
Consequently, we get $\{e_{x}(n+1)\leq j\} \in \F_{j}^{x}$ for all $j \in \mathbb{N}$ and therefore also $\{e_{x}(n+1)= j\} \in \F_{j}^{x}$.
\end{proof}

\begin{lemma}\label{lemma_ineqx} Let $x:\Omega \rightarrow V$ be an independent initial leading to extinction. Then the following assertions hold.
\begin{enumerate}
	\item $\x_{x,m}\in V_{i}$ for all $m \in \mathbb{N}_{0}$ and $i \in \{1,2\}$ almost surely. 
	\item $||\x_{x,m}||_{V_{1}}^{\rho} \leq (-\kappa \beta_{m}+||\x_{x,m-1}||_{V_{1}}^{\rho})_{+}+||\eta_{m}||_{V_{1}}^{\rho}$ for all $m \in \mathbb{N}$ almost surely.
\end{enumerate}
\end{lemma}
\begin{proof} Let us start by proving i) inductively. The result is trivial for $m=0$. So assume it holds for an $m \in \mathbb{N}$. By the induction hypothesis, we have $\x_{x,m}\in V_{i}$ a.s., which gives $T(\beta_{m+1})\x_{x,m}\in V_{i}$ a.s. by Assumption \ref{assumption_fe}.i). As also $\eta_{m+1}\in V_{i}$ a.s. by Assumption \ref{assumption_mb}.i), we get $\x_{x,m+1}=T(\beta_{m+1})\x_{x,m}+\eta_{m+1}\in V_{i}$ almost surely and i) follows.\\
Now, let us prove ii). Appealing to Assumption \ref{assumption_fe}.ii), while having in mind i), gives 
\begin{align*}
||\x_{x,m}||_{V_{1}}^{\rho} \leq ||T(\beta_{m})\x_{x,m-1}||_{V_{1}}^{\rho}+||\eta_{m}||_{V_{1}}^{\rho} \leq (-\kappa \beta_{m}+||\x_{x,m-1}||_{V_{1}}^{\rho})_{+}+||\eta_{m}||_{V_{1}}^{\rho} 
\end{align*}
for all $m \in \mathbb{N}$ almost surely, where the well-known inequality $(a+b)^{\rho}\leq a^{\rho}+b^{\rho}$ for all $a,b\geq 0$, was used.
\end{proof}

\begin{lemma}\label{lemma_inclusionx} Let $x:\Omega \rightarrow V$ be an independent initial leading to extinction and introduce $m,n\in\mathbb{N}$, with $m<n$. Then the inclusion
\begin{align*}
\{-\kappa \beta_{k}+||\x_{x,k-1}||_{V_{1}}^{\rho}>0,~\forall k=m,...,n\} \subseteq \{-\kappa \sum \limits_{k=m}\limits^{n}\beta_{k}+\sum \limits_{k=m}\limits^{n-1}||\eta_{k}||_{V_{1}}^{\rho}+||\x_{x,m-1}||_{V_{1}}^{\rho}>0\}
\end{align*} 
holds up to a $\P$-null-set.
\end{lemma}
\begin{proof} Fix $n \in \mathbb{N}\setminus \{1\}$ and let us prove inductively that
\begin{align}
\label{lemma_inclusionproofeq1}
\{-\kappa \beta_{k}+||\x_{x,k-1}||_{V_{1}}^{\rho}>0,\forall k=n-j,...,n\} \subseteq \{-\kappa \sum \limits_{k=n-j}\limits^{n}\beta_{k}+\sum \limits_{k=n-j}\limits^{n-1}||\eta_{k}||_{V_{1}}^{\rho}+||\x_{x,n-j-1}||_{V_{1}}^{\rho}>0\}
\end{align}
for all $j=1,...,n-1$ almost surely, which obviously yields the claim.\\ 
So let $j=1$. Firstly, invoking Lemma \ref{lemma_ineqx}.ii) gives $||\x_{x,n-1}||_{V_{1}}^{\rho}\leq (-\kappa \beta_{n-1}+||\x_{x,n-2}||_{V_{1}}^{\rho})_{+}+||\eta_{n-1}||_{V_{1}}^{\rho}$ a.s. and therefore
\begin{align}
\label{lemma_inclusionproofeq2}
\{-\kappa \beta_{n}+||\x_{x,n-1}||_{V_{1}}^{\rho}>0\}\subseteq \{-\kappa \beta_{n}+(-\kappa \beta_{n-1}+||\x_{x,n-2}||_{V_{1}}^{\rho})_{+}+||\eta_{n-1}||_{V_{1}}^{\rho}>0\}
\end{align}
almost surely. Using this yields
\begin{eqnarray*}
	& & ~
	\{-\kappa \beta_{k}+||\x_{x,k-1}||_{V_{1}}^{\rho}>0,~\forall k=n-1,...,n\} \\
	& \subseteq & ~ \{-\kappa \beta_{n-1}+||\x_{x,n-2}||_{V_{1}}^{\rho}>0,~-\kappa \beta_{n}+(-\kappa \beta_{n-1}+||\x_{x,n-2}||_{V_{1}}^{\rho})_{+}+||\eta_{n-1}||_{V_{1}}^{\rho}>0\}\\
    & \subseteq & ~  \{-\kappa \sum \limits_{k=n-1}\limits^{n}\beta_{k}+\sum \limits_{k=n-1}\limits^{n-1}||\eta_{k}||_{V_{1}}^{\rho}+||\x_{x,n-2}||_{V_{1}}^{\rho}>0\}
\end{eqnarray*}
almost surely, and consequently (\ref{lemma_inclusionproofeq1}) holds for $j=1$.\\
Now assume (\ref{lemma_inclusionproofeq1}) holds for a $j \in \{1,...,n-2\}$ (and w.l.o.g. that $n>2$). Firstly, using the induction hypothesis yields
\begin{eqnarray*}
& & ~
\{-\kappa \beta_{k}+||\x_{x,k-1}||_{V_{1}}^{\rho}>0,\forall k=n-(j+1),...,n\}\\
& \subseteq & ~ \{-\kappa \sum \limits_{k=n-j}\limits^{n}\beta_{k}+\sum \limits_{k=n-j}\limits^{n-1}||\eta_{k}||_{V_{1}}^{\rho}+||\x_{x,n-j-1}||_{V_{1}}^{\rho}>0\} \cap \{-\kappa \beta_{n-j-1}+||\x_{x,n-j-2}||_{V_{1}}^{\rho}>0\}
\end{eqnarray*}
almost surely. Appealing to Lemma \ref{lemma_ineqx}.ii) once more, yields
\begin{eqnarray*}
& & ~
\{-\kappa \sum \limits_{k=n-j}\limits^{n}\beta_{k}+\sum \limits_{k=n-j}\limits^{n-1}||\eta_{k}||_{V_{1}}^{\rho}+||\x_{x,n-j-1}||_{V_{1}}^{\rho}>0\}\\
& \subseteq & ~ \{-\kappa \sum \limits_{k=n-j}\limits^{n}\beta_{k}+\sum \limits_{k=n-j}\limits^{n-1}||\eta_{k}||_{V_{1}}^{\rho}+(-\kappa \beta_{n-j-1}+||\x_{x,n-j-2}||^{\rho}_{V_{1}})_{+}+||\eta_{n-j-1}||_{V_{1}}^{\rho}>0\}
\end{eqnarray*} 
almost surely. Finally, combining the former and the latter inclusion gives the claim.
\end{proof}

\begin{proposition}\label{prop_extinction} Let $x:\Omega \rightarrow V$ be an independent initial leading to extinction. Then we have
	\begin{align*}
	\P(e_{x}(i)<\infty,~\forall i \in \mathbb{N})=1.
	\end{align*}
\end{proposition}
\begin{proof} It obviously suffices to prove that $e_{x}(i)<\infty$ a.s. for all $i \in \mathbb{N}$. This will be proven inductively.\\
Firstly, employing the $\sigma$-continuity of probability measures from above yields
\begin{align*}
\P(e_{x}(1)=\infty)=\lim \limits_{n \rightarrow \infty} \P(T(\beta_{k})\x_{x,k-1}\neq 0,~\forall k=1,...,n).
\end{align*}
Moreover, appealing to Lemma \ref{lemma_ineqx}.i) gives $\x_{x,k-1} \in V_{1}$ for all $k \in \mathbb{N}$ a.s. Consequently, Assumption \ref{assumption_fe}.ii) gives
\begin{align}
\label{prop_extinctionproofeq0}
\{T(\beta_{k})\x_{x,k-1}\neq 0\} \subseteq \{(-\kappa\beta_{k}+||\x_{x,k-1}||_{V_{1}}^{\rho})_{+}> 0\} = \{-\kappa\beta_{k}+||\x_{x,k-1}||_{V_{1}}^{\rho}> 0\},~\forall k \in \mathbb{N}
\end{align}
a.s. Using this, while having in mind Lemma \ref{lemma_inclusionx} yields
\begin{align}
\label{prop_extinctionproofeq1}
\P(e_{x}(1)=\infty) \leq \lim \limits_{n \rightarrow \infty} \P\left(-\kappa \sum \limits_{k=1}\limits^{n}\beta_{k}+\sum \limits_{k=1}\limits^{n-1}||\eta_{k}||_{V_{1}}^{\rho}+||x||_{V_{1}}^{\rho}>0\right).
\end{align}
Now note that $||x||_{V_{1}}^{\rho},\beta_{k},~||\eta_{k}||_{V_{1}}^{\rho}\in L^{2}(\Omega)$. Consequently, we can introduce
\begin{align*}
\nu_{n}:= \mathbb{E} \left(-\kappa \sum \limits_{k=1}\limits^{n}\beta_{k}+\sum \limits_{k=1}\limits^{n-1}||\eta_{k}||_{V_{1}}^{\rho}+||x||_{V_{1}}^{\rho}\right)= n\left(-\kappa \mathbb{E}(\beta_{1})+\mathbb{E}||\eta_{1}||_{V_{1}}^{\rho}\right)-\mathbb{E}||\eta_{1}||_{V_{1}}^{\rho}+\mathbb{E}||x||_{V_{1}}^{\rho}.
\end{align*}
Moreover, appealing to Assumption \ref{assumption_mb}.iii) yields $\nu_{n}<0$ for all $n$ sufficiently large. Consequently, by invoking (\ref{prop_extinctionproofeq1}) and employing Tschebyscheff's inequality, we get
\begin{eqnarray*}
\P(e_{x}(1)=\infty)
& \leq & \lim \limits_{n \rightarrow \infty}\frac{1}{\nu_{n}^{2}} \Var\left(-\kappa \sum \limits_{k=1}\limits^{n}\beta_{k}+\sum \limits_{k=1}\limits^{n-1}||\eta_{k}||_{V_{1}}^{\rho}+||x||_{V_{1}}^{\rho}\right) \\
& = & \lim \limits_{n \rightarrow \infty}\frac{1}{\nu_{n}^{2}} \left(\kappa^{2}\Var(\beta_{1})n +(n-1)\Var(||\eta_{1}||_{V_{1}}^{\rho})+(\Var||x||_{V_{1}}^{\rho})\right) \\
& = & ~ 0,
\end{eqnarray*}
which proves $\P(e_{x}(1)<\infty)=1$.\\
Now assume $e_{x}(i)<\infty$ a.s. for a given $i \in \mathbb{N}$. Then there is a set $M_{i}\subseteq \mathbb{N}$, such that $\P(e_{x}(i)\in M_{i})=1$ and $\P(e_{x}(i)=m)>0$ for all $m \in M_{i}$. This implies
\begin{align*}
\P(e_{x}(i+1)=\infty)=\sum \limits_{m\in M_{i}}\P(e_{x}(i+1)=\infty,~e_{x}(i)=m).
\end{align*}
Consequently, it suffices to prove that $\P(e_{x}(i+1)=\infty,~e_{x}(i)=m)=0$ for all $m \in \mathbb{M}_{i}$. So let $m \in M_{i}$ be given. Then we have
\begin{align*}
\P(e_{x}(i+1)=\infty,~e_{x}(i)=m) = \P(T(\beta_{k})\x_{x,k-1}\neq 0,~\forall k>m,~e_{x}(i)=m).
\end{align*}
Consequently, employing the $\sigma$-continuity of probability measures, (\ref{prop_extinctionproofeq0}) and Lemma \ref{lemma_inclusionx} gives 
\begin{align*}
\P(e_{x}(i+1)=\infty,~e_{x}(i)=m) \leq  \lim \limits_{n  \rightarrow \infty}\P\left(-\kappa \sum \limits_{k=m+1}\limits^{n}\beta_{k}+\sum \limits_{k=m+1}\limits^{n-1}||\eta_{k}||_{V_{1}}^{\rho}+||\x_{x,m}||_{V_{1}}^{\rho}>0,~e_{x}(i)=m\right)
\end{align*}
Moreover, it is plain that $\x_{x,m}=\eta_{m}$ on $\{e_{x}(i)=m\}$ which implies
\begin{eqnarray*}
\P(e_{x}(i+1)=\infty,~e_{x}(i)=m) 
& \leq & ~  \lim \limits_{n  \rightarrow \infty}\P\left(-\kappa \sum \limits_{k=m+1}\limits^{n}\beta_{k}+\sum \limits_{k=m+1}\limits^{n-1}||\eta_{k}||_{V_{1}}^{\rho}+||\eta_{m}||_{V_{1}}^{\rho}>0\right)\\
& = & ~\lim \limits_{n  \rightarrow \infty}\P\left(-\kappa \sum \limits_{k=1}\limits^{n-m}\beta_{k}+\sum \limits_{k=1}\limits^{n-m}||\eta_{k}||_{V_{1}}^{\rho}>0\right),
\end{eqnarray*}
where the last equality follows from the fact that the $\eta_{k}$'s as well as the $\beta_{k}$'s are i.i.d. and independent of each other. Analogously to the induction beginning, one now easily verifies by the aid of Tschebyscheff's inequality that the last limit converges to zero and the claim follows.
\end{proof}

\begin{remark}\label{lemma_xichoice} The following observations will be useful in the sequel. The easy proofs are left to the reader.
	\begin{enumerate}
		\item If $(\Xi,(W,||\cdot||_{W}))\in SL_{V_{2}}(V)$, then $(||\Xi||_{W},\mathbb{R})\in SL_{V_{2}}(V)$.
		\item If $(\Xi,(W,||\cdot||_{W}))\in SL_{V_{2}}(V)$ and $w \in W$, then $(\Xi_{w},(W,||\cdot||_{W}))\in SL_{V_{2}}(V)$,  where we set $\Xi_{w}(v):=\Xi(v)+w$ for all $v \in V$.
	\end{enumerate}
\end{remark}

\begin{lemma}\label{lemma_Xbasicprop} Let $(\Xi,(W,||\cdot||_{W}))\in SL_{V_{2}}(V)$ and let $x:\Omega \rightarrow V$ be an independent initial leading to extinction. Then we have
\begin{enumerate}
	\item $\P(\X_{x}(t)\in V_{i},~\forall t \geq 0)=1$, where $i\in \{1,2\}$.
	\item  The mapping defined by $[0,\infty) \times \Omega\ni(t,\omega)\mapsto \Xi(\X_{x}(t,\omega))$ is  $\B([0,\infty))\otimes\F $-$\B(W)$-measurable.
	\item $\P\left(\int \limits_{0}\limits^{t}||\Xi(\X_{x}(\tau))||_{W}d\tau<\infty,~\forall t \geq 0\right)=1$.
\end{enumerate}
Consequently, the Bochner integral $\int \limits_{0}\limits^{t}\Xi(\X_{x}(\tau))d\tau$ is (up-to a $\P$-null-set which is independent of $t$) well-defined, for all $t\geq 0$, and the stochastic process defined by $[0,\infty)\times \Omega \ni (t,w)\mapsto \int \limits_{0}\limits^{t}\Xi(\X_{x}(\tau,\omega))d\tau$ is $\F \otimes \B([0,\infty))$-$\B(W)$-measurable.
\end{lemma}
\begin{proof} Appealing to Lemma \ref{lemma_ineqx}.i) yields the existence of a $\P$-null-set $M_{1} \in \F$ such that
\begin{align}
\label{lemma_v2inv}
\x_{x,m}(\omega)\in V_{i},~\forall \omega \in \Omega\setminus M_{1},~m \in \mathbb{N}_{0},~i\in\{1,2\}.
\end{align}
Moreover, there is a $\P$-null-set $M_{2} \in \F$ such that $\lim \limits_{n \rightarrow \infty}\alpha_{n}(\omega)=\infty$ and $\alpha_{m}(\omega)<\alpha_{m+1}(\omega)$ for all $m \in \mathbb{N}_{0}$ and $\omega \in \Omega \setminus M_{2}$. Introduce $M:=M_{1}\cup M_{2}$.\\ 
Proof of i). For any fixed $\omega \in \Omega \setminus M$ and $t \in [0,\infty)$, there is an $m \in \mathbb{N}_{0}$ such that $t \in [\alpha_{m}(\omega),\alpha_{m+1}(\omega))$, which yields
\begin{align}
\label{lemma_v2inv2}
\X_{x}(t,\omega)=T(t-\alpha_{m}(\omega))\x_{x,m}(\omega) \in V_{i},~\forall i \in \{1,2\}
\end{align}
by Assumption \ref{assumption_fe}.i) and (\ref{lemma_v2inv}). Consequently, i) holds.\\
Moreover, as $\X_{x}$ is $\B([0,\infty))\otimes\F $-$\B(V)$-measurable (see Remark \ref{remark_Xmeascad}) and $\Xi$ is $\B(V)$-$\B(W)$-measurable, ii) holds as well.\\
Proof of iii). Let $t>0$ and $\omega \in \Omega \setminus M$ be arbitrary but fixed. Moreover, introduce $m \in \mathbb{N}$ such that $t<\alpha_{m}(\omega)$. Then (\ref{lemma_v2inv2}) enables us to conclude that there are constants $c_{1},~c_{2}\in [0,\infty)$ such that
\begin{eqnarray*}
	\int \limits_{0}\limits^{t} ||\Xi(\X_{x}(\tau,\omega))||_{W}d\tau
	& \leq & ~\sum \limits_{k=0}\limits^{m-1} \int \limits_{\alpha_{k}(\omega)}\limits^{\alpha_{k+1}(\omega)} ||\Xi(T(\tau-\alpha_{k}(\omega))\x_{x,k}(\omega))||_{W}d\tau \\
	& \leq & ~\sum \limits_{k=0}\limits^{m-1} \int \limits_{\alpha_{k}(\omega)}\limits^{\alpha_{k+1}(\omega)} c_{1}||T(\tau-\alpha_{k}(\omega))\x_{x,k}(\omega)||_{V_{2}}d\tau +c_{2}\beta_{k+1}(\omega)\\
	& \leq & ~ \sum \limits_{k=0}\limits^{m-1} \beta_{k+1}(\omega)(c_{1}||\x_{x,k}(\omega)||_{V_{2}}+c_{2}),
\end{eqnarray*}	
where the last inequality follows from Assumption \ref{assumption_fe}.iii). Consequently, iii) is proven, since the $\P$-null-set $M$ is indeed independent of $t\geq 0$.\\
Moreover, it follows from ii) that $[0,\infty) \ni t \mapsto \Xi(\X_{x}(\omega,t))$ is $\B([0,\infty))$-$\B(W)$-measurable for all $\omega \in \Omega$. This and (the proof of) iii) yields that the Bochner integral $\int \limits_{0}\limits^{t}\Xi(\X_{x}(\tau,\omega))d\tau$ exists for all $\omega \in \Omega \setminus M$ and $t \geq 0$.\\
Finally, \cite[Lemma 2.2.4]{SIBS} yields that $[0,\infty)\times \Omega \ni (t,\omega)\mapsto \int \limits_{0}\limits^{t}\Xi(\X_{x}(\tau,\omega))d\tau:=I(t,\omega)$ is (almost surely) continuous and that each $I(t)$ is $\F$-$\B(W)$-measurable. This implies that $I$ is $\F \otimes \B([0,\infty))$-$\B(W)$-measurable, by \cite[Proposition 2.2.3]{SIBS}. (The results in \cite{SIBS} are formulated for filtered probability spaces, chose the filtration which is constantly $\F$ while applying \cite[Lemma 2.2.4, Proposition 2.2.3]{SIBS}.) 
\end{proof}

The preceding lemma yields in particular that $\Omega \ni \omega \mapsto \int \limits_{a_{1}(\omega)}\limits^{a_{2}(\omega)}\Xi(\X_{x}(\omega,\tau))d\tau$ is well-defined and $\F$-$\B(W)$-measurable, whenever $a_{i}:\Omega \rightarrow [0,\infty)$ are $\F$-$\B([0,\infty))$-measurable, $x:\Omega \rightarrow V$ is an initial leading to extinction and $(\Xi,(W,||\cdot||_{W}))\in SL_{V_{2}}(V)$.\\
Our next goal is to establish that the sequence defined by $\left(\int \limits_{\alpha_{e_{x}(n)}}\limits^{\alpha_{e_{x}(n+1)}}\Xi(\X_{x}(\tau))d\tau\right)_{n \in \mathbb{N}}$ is i.i.d.

\begin{remark}\label{remark_stoppedfiltration}  Whenever $x:\Omega \rightarrow V$ is an independent initial leading to extinction, then $(\F^{x}_{e_{x}(n)})_{n \in \mathbb{N}}$ denotes the stopped filtration, defined by
\begin{align*}
\F^{x}_{e_{x}(n)}:=\{A\in \F:~A\cap \{e_{x}(n)=j\}\in \F^{x}_{j},~\forall j \in \mathbb{N}\},
\end{align*}
for all $n \in \mathbb{N}$.\\
Note that $(\F^{x}_{j})_{j \in \mathbb{N}}$ is trivially a filtration. Moreover, invoking Lemma \ref{lemma_xbasicprop}.iii) yields that each $e_{x}(n)$ is a stopping time w.r.t. $(\F^{x}_{j})_{j\in \mathbb{N}}$ and that $e_{x}(n)\leq e_{x}(n+1)$ for all $n \in \mathbb{N}$. Consequently, it is standard that each $\F^{x}_{e_{x}(n)}$ is indeed a $\sigma$-algebra and that $\F^{x}_{e_{x}(n)} \subseteq \F^{x}_{e_{x}(n+1)}$ for all $n \in \mathbb{N}$. In addition, it is plain that $(\F^{x}_{e_{x}(n)})_{n \in \mathbb{N}}$ inherits the completeness of $(\F^{x}_{j})_{j\in \mathbb{N}}$.
\end{remark}

\begin{remark}\label{remark_xbar} In all that follows $\overline{x}\in \mathcal{M}(\Omega;V)$, denotes a mapping fulfilling
\begin{enumerate}
	\item $\overline{x}=\eta_{1}$ in distribution and
	\item $\overline{x}$ is jointly independent of $(\eta_{m})_{m\in \mathbb{N}}$ and $(\beta_{m})_{m \in \mathbb{N}}$.
\end{enumerate}
Note that this implies $\overline{x}\in V_{i}$ a.s. for $i=1,2$. Moreover, as $0<2\rho<\rho(11+\hat{\varepsilon})$, we also have $\mathbb{E}||\overline{x}||_{V_{1}}^{2\rho}<\infty$, which gives that $\overline{x}$ is an independent initial leading to extinction.
\end{remark}

\begin{lemma}\label{lemma_iidpreparation} Let $(\Xi,(W,||\cdot||_{W}))\in SL_{V_{2}}(V)$ and $x:\Omega \rightarrow V$ be an independent initial leading to extinction. Then we have
\begin{align*}
\mathbb{E}\left(f\left(\int \limits_{\alpha_{e_{x}(n)}}\limits^{\alpha_{e_{x}(n+1)}}\Xi(\X_{x}(\tau))d\tau\right)\Big\vert \F^{x}_{e_{x}(n)} \right)= \mathbb{E}f\left(\int \limits_{0}\limits^{\alpha_{e_{\overline{x}}(1)}}\Xi(\X_{\overline{x}}(\tau))d\tau\right),
\end{align*}
for all $n \in \mathbb{N}$ and $f:W\rightarrow \mathbb{R}$ which are $\B(W)$-$\B(\R)$-measurable and bounded.
\end{lemma}
\begin{proof} Let $A \in \F^{x}_{e_{x}(n)}$ be given and introduce $A_{i}:=\{\omega \in A:~e_{x}(n)(\omega)=i\}$ for all $i \in \mathbb{N}$, with $i\geq n$. At first, it will be shown that
\begin{align}
\label{lemma_iidpreparationproofeq1}
\mathbb{E}\id_{A_{i}}\hat{f}_{j}(\x_{x,i},...,\x_{x,i+j-1},\beta_{i+1},...,\beta_{i+j}) =\P(A_{i} ) \mathbb{E}\hat{f}_{j}(\x_{\overline{x},0},...,\x_{\overline{x},j-1},\beta_{1},...,\beta_{j}),
\end{align}
for all $i \in \mathbb{N}$, with $i\geq n$, all $j \in \mathbb{N}$ and $\hat{f}_{j}:V^{j}\times [0,\infty)^{j}\rightarrow \mathbb{R}$ which are bounded and $\B(V^{j})\otimes \B([0,\infty)^{j})$-$\B(\R)$-measurable.\\ 
Now let us prove (\ref{lemma_iidpreparationproofeq1}) inductively w.r.t. $j \in \mathbb{N}$.\\
Let $j=1$, $i\geq n$ and $\hat{f}_{1}:V\times[0,\infty)\rightarrow \mathbb{R}$ be bounded and measurable. Note that $T(\beta_{i})\x_{x,i-1}=0$ on $A_{i}$. Consequently, we get $\mathbb{E}\id_{A_{i}} \hat{f}_{1}(\x_{x,i},\beta_{i+1}) = \mathbb{E}\id_{A_{i}} \hat{f}_{1}(\eta_{i},\beta_{i+1})$.
Moreover, appealing to Remark \ref{remark_stoppedfiltration} yields that $A_{i}\in \F_{i}^{x}=\sigma_{0}(x,\beta_{1},...,\beta_{i},\eta_{1},...,\eta_{i-1})$. Hence, $A_{i}$ is independent of $\hat{f}_{1}(\eta_{i},\beta_{i+1})$, which gives 
\begin{align*}
\mathbb{E}\id_{A_{i}} \hat{f}_{1}(\x_{x,i},\beta_{i+1}) =\P(A_{i})\mathbb{E}\hat{f}_{1}(\eta_{i},\beta_{i+1})=\P(A_{i})\mathbb{E}\hat{f}_{1}(\x_{\overline{x},0},\beta_{1}), 
\end{align*}
where the last inequality follows from the fact that $(\x_{\overline{x},0},\beta_{1})=(\overline{x},\beta_{1})$, which is in distribution equal to $(\eta_{i},\beta_{i+1})$. Hence, (\ref{lemma_iidpreparationproofeq1}) holds for $j=1$.\\
Now assume that it holds for an $j \in \mathbb{N}$, let $i\geq n$ and $\hat{f}_{j+1}:V^{j+1}\times[0,\infty)^{j+1}\rightarrow \mathbb{R}$ be bounded and $\B(V^{j+1})\otimes \B([0,\infty)^{j+1})$-$\B(\R)$-measurable. Moreover, for any $\tilde{\beta}\in [0,\infty),~\tilde{\eta}\in V$, introduce \linebreak$\hat{f}_{\tilde{\beta},\tilde{\eta}}:V^{j}\times [0,\infty)^{j}\rightarrow \mathbb{R}$, by
\begin{align*}
\hat{f}_{\tilde{\beta},\tilde{\eta}}(y_{0},...,y_{j-1},b_{1},...,b_{j}):= \hat{f}_{j+1}(y_{0},...,y_{j-1},T_{\A}(b_{j})y_{j-1}+\tilde{\eta},b_{1},...,b_{j},\tilde{\beta}),
\end{align*}
for all $y_{0},...,y_{j-1},\tilde{\eta}\in V$ and $b_{1},...,b_{j},\tilde{\beta}\in [0,\infty)$. Then $\hat{f}_{\tilde{\beta},\tilde{\eta}}$ inherits the boundedness of $\hat{f}_{j+1}$. Moreover, invoking Remark \ref{remark_measuc0sg}, gives that $\hat{f}_{\tilde{\beta},\tilde{\eta}}$ is $\B(V^{j})\otimes \B([0,\infty)^{j})$-$\B(\mathbb{R})$-measurable, as it is the composition of measurable functions, for all $\tilde{\beta}\in [0,\infty)$ and $\tilde{\eta}\in V$. Consequently, the induction hypothesis yields
\begin{align*}
\mathbb{E}\id_{A_{i}}\hat{f}_{\tilde{\beta},\tilde{\eta}}(\x_{x,i},...,\x_{x,i+j-1},\beta_{i+1},...,\beta_{i+j})d\P =\P(A_{i} ) \mathbb{E}\hat{f}_{\tilde{\beta},\tilde{\eta}}(\x_{\overline{x},0},...,\x_{\overline{x},j-1},\beta_{1},...,\beta_{j}),
\end{align*}
which gives
\begin{eqnarray*}
	& & ~
\mathbb{E}\id_{A_{i}}\hat{f}_{j+1}(\x_{x,i},...,\x_{x,i+j-1},T_{\A}(\beta_{i+j})\x_{x,i+j-1}+\tilde{\eta},\beta_{i+1},...,\beta_{i+j},\tilde{\beta})\\
& = & ~\P(A_{i} ) \mathbb{E}\hat{f}_{j+1}(\x_{\overline{x},0},...,\x_{\overline{x},j-1},T_{\A}(\beta_{j})\x_{\overline{x},j-1}+\tilde{\eta},\beta_{1},...,\beta_{j},\tilde{\beta}),
\end{eqnarray*}
for all $i \geq n$, $\tilde{\beta}\in [0,\infty)$ and $\tilde{\eta}\in V$.\\
Moreover, Lemma \ref{lemma_xbasicprop} yields, that $(\x_{x,i},...,\x_{x,i+j-1},\beta_{i+1},...,\beta_{i+j})$ is $\F^{x}_{i+j}$-$\B(V^{j})\otimes \B([0,\infty)^{j})$-measurable and, a fortiori, that $\id_{A_{i}}$ is $\F^{x}_{i+j}$-$\B(\mathbb{R})$-measurable. Consequently, as $(\beta_{i+j+1}, \eta_{i+j})$ is independent of $\F^{x}_{i+j}$ and as $(\beta_{i+j+1},\eta_{i+j})=(\beta_{j+1},\eta_{j})$ in distribution, we get 
\begin{eqnarray*}
& & ~
\mathbb{E}\id_{A_{i}}\hat{f}_{j+1}(\x_{x,i},...,\x_{x,i+j},\beta_{i+1},...,\beta_{i+j+1})\\
& = & ~ \mathbb{E}(\id_{A_{i}}\hat{f}_{j+1}(\x_{x,i},...,\x_{x,i+j-1},T_{\A}(\beta_{i+j})\x_{x,i+j-1}+\eta_{i+j},\beta_{i+1},...,\beta_{i+j},\beta_{i+j+1}))\\
& = & ~ \int \limits_{[0,\infty)\times V} \mathbb{E}(\id_{A_{i}}\hat{f}_{j+1}(\x_{x,i},...,\x_{x,i+j-1},T_{\A}(\beta_{i+j})\x_{x,i+j-1}+\tilde{\eta},\beta_{i+1},...,\beta_{i+j},\tilde{\beta}))d\P_{(\beta_{i+j+1},\eta_{i+j})}(\tilde{\beta},\tilde{\eta})\\
& = & ~ \int \limits_{[0,\infty)\times V} \P(A_{i} ) \mathbb{E}\hat{f}_{j+1}(\x_{\overline{x},0},...,\x_{\overline{x},j-1},T_{\A}(\beta_{j})\x_{\overline{x},j-1}+\tilde{\eta},\beta_{1},...,\beta_{j},\tilde{\beta})d\P_{(\beta_{i+j+1},\eta_{i+j})}(\tilde{\beta},\tilde{\eta})\\
& = & ~ \int \limits_{[0,\infty)\times V} \P(A_{i} ) \mathbb{E}\hat{f}_{j+1}(\x_{\overline{x},0},...,\x_{\overline{x},j-1},T_{\A}(\beta_{j})\x_{\overline{x},j-1}+\tilde{\eta},\beta_{1},...,\beta_{j},\tilde{\beta})d\P_{(\beta_{j+1},\eta_{j})}(\tilde{\beta},\tilde{\eta}).
\end{eqnarray*}
Now, appealing to Lemma \ref{lemma_xbasicprop} yields, that $(\x_{\overline{x},0},...,\x_{\overline{x},j-1},\beta_{1},...,\beta_{j})$ is $\F^{\overline{x}}_{j}$-$\B(V^{j})\otimes [0,\infty)^{j}$-measurable. (Note that this is indeed possible, since $\overline{x}$ is also an independent initial leading to extinction, see Remark \ref{remark_xbar}.) Moreover, it is plain that $(\beta_{j+1},\eta_{j})$ is independent of $\F^{\overline{x}}_{j}$. Consequently, we get
\begin{align*}
\mathbb{E}\id_{A_{i}}\hat{f}_{j+1}(\x_{x,i},...,\x_{x,i+j},\beta_{i+1},...,\beta_{i+j+1})=\P(A_{i}) \mathbb{E}\hat{f}_{j+1}(\x_{\overline{x},0},...,\x_{\overline{x},j-1},\x_{\overline{x},j},\beta_{1},...,\beta_{j},\beta_{j+1}),
\end{align*}
which gives (\ref{lemma_iidpreparationproofeq1}).\\
Now the actual claim is proven by the aid of (\ref{lemma_iidpreparationproofeq1}). Firstly, appealing to Lemma \ref{lemma_xbasicprop}.ii) and Proposition \ref{prop_extinction} yields
\begin{align*}
\mathbb{E}\left(\id_{A}f\left(\int \limits_{\alpha_{e_{x}(n)}}\limits^{\alpha_{e_{x}(n+1)}}\Xi(\X_{x}(\tau))d\tau\right) \right) = \sum \limits_{i=n}\limits^{\infty}\sum \limits_{j=1}\limits^{\infty}\mathbb{E}\left(\id_{A_{i}}\id_{\{e_{x}(n+1)=i+j\}}f\left(\int \limits_{\alpha_{i}}\limits^{\alpha_{i+j}}\Xi(\X_{x}(\tau))d\tau\right) \right).
\end{align*}
In addition, we have
\begin{align*}
\int \limits_{\alpha_{i}}\limits^{\alpha_{i+j}}\Xi(\X_{x}(\tau))d\tau =  \sum \limits_{k=i}\limits^{i+j-1}\int \limits_{0}\limits^{\beta_{k+1}}\Xi(T_{\A}(\tau)\x_{x,k})d\tau= \sum \limits_{k=0}\limits^{j-1}\int \limits_{0}\limits^{\beta_{k+i+1}}\Xi(T_{\A}(\tau)\x_{x,k+i})d\tau
\end{align*}
Combining the former and the latter  equality implies
\begin{align*}
\mathbb{E}\left(\id_{A}f\left(\int \limits_{\alpha_{e_{x}(n)}}\limits^{\alpha_{e_{x}(n+1)}}\Xi(\X_{x}(\tau))d\tau\right) \right) = \sum \limits_{i=n}\limits^{\infty}\sum \limits_{j=1}\limits^{\infty}\mathbb{E}\left(\id_{A_{i}}\id_{\{e_{x}(n+1)=i+j\}}f\left( \sum \limits_{k=0}\limits^{j-1}\int \limits_{0}\limits^{\beta_{k+i+1}}\Xi(T_{\A}(\tau)\x_{x,k+i})d\tau\right) \right).
\end{align*}
For all $j \in \mathbb{N}$, introduce $\hat{h}_{j}:V^{j}\times [0,\infty)^{j}\times \mathbb{R}$, by
\begin{align*} 
\hat{h}_{j}(y_{0},...,y_{j-1},b_{1},...,b_{j}):=f\left( \sum \limits_{k=0}\limits^{j-1}\int \limits_{0}\limits^{b_{k+1}}\Xi(T_{\A}(\tau)y_{k})\id_{V_{2}}(y_{k})d\tau\right).
\end{align*}
Invoking Remark \ref{remark_measuc0sg} and Remark \ref{lemma_meas1}, gives that $[0,\infty)\times V \ni (\tau,y)\mapsto \Xi(T_{\A}(\tau)y)\id_{V_{2}}(y)$ is $\B([0,\infty))\otimes V$-$\B(W)$-measurable. Moreover, working as in the proof of Lemma \ref{lemma_Xbasicprop} yields that $\Xi(T_{\A}(\cdot)y)\id_{V_{2}}(y)\in L^{1}([0,t];W)$ for all $t>0$ and $y\in V$. Consequently, \cite[Proposition 2.1.3]{SIBS} yields that $(y,t)\mapsto \int \limits_{0}\limits^{t}\Xi(T_{\A}(\tau)y)\id_{V_{2}}(y)d\tau$ is, for each $y$, as mapping in $t$ continuous, and by \cite[Proposition 2.1.4]{SIBS} it is for each $t \in [0,\infty)$, as a mapping in $y$, $\B(V)$-$\B(W)$-measurable. Consequently, this mapping is $\B(V)\otimes \B([0,\infty))$-$\B(W)$-measurable, see \cite[Lemma 4.51]{IDA}.\\ 
Using these observations, it is plain to deduce that $\hat{h}_{j}$ is $\B(V^{j})$-$\B([0,\infty)^{j})$-$\B(\mathbb{R})$-measurable for all $j \in \mathbb{N}$. Moreover, each $\hat{h}_{j}$ is obviously bounded.\\
For all $j \in \mathbb{N}$, introduce $\hat{g}_{j}:V^{j}\times [0,\infty)^{j}\times \mathbb{R}$, by
\begin{align*}
\hat{g}_{j}(y_{0},...,y_{j-1},b_{1},...,b_{j}):=\id\{T_{\A}(b_{k})y_{k-1}\neq 0, \forall k=1,...,j-1,~ T_{\A}(b_{j})y_{j-1}=0\}, ~\forall j \in \mathbb{N}\setminus\{1\}
\end{align*}
and $\hat{g}_{1}(y_{0},b_{1}):=\id\{ T_{\A}(b_{1})y_{0}=0\}$. Then $\hat{g}_{j}$ is obviously bounded, and by the aid of Remark \ref{remark_measuc0sg} also $\B(V^{j})\otimes \B([0,\infty)^{j})$-$\B(\R)$-measurable.\\
Moreover, appealing to Lemma \ref{lemma_ineqx}.i) yields
\begin{align*}
\hat{h}_{j}(\x_{x,i},...,\x_{x,i+j-1},\beta_{i+1},...,\beta_{i+j})=f\left( \sum \limits_{k=0}\limits^{j-1}\int \limits_{0}\limits^{\beta_{i+k+1}}\Xi(T_{\A}(\tau)\x_{x,i+k})d\tau\right),~\forall i \geq n,~j \in \mathbb{N}
\end{align*}
almost surely. In addition, for all $\omega \in A_{i}$, we have
\begin{align*}
\hat{g}_{j}(\x_{x,i},...,\x_{x,i+j-1},\beta_{i+1},...,\beta_{i+j})(\omega)=\id_{\{e_{x}(n+1)=i+j\}}(\omega),~\forall i \geq n,~j\in \mathbb{N}.
\end{align*}
Consequently, putting it all together yields
\begin{eqnarray*}
	& & ~
	\mathbb{E}\left(\id_{A}f\left(\int \limits_{\alpha_{e_{x}(n)}}\limits^{\alpha_{e_{x}(n+1)}}\Xi(\X_{x}(\tau))d\tau\right) \right)\\
	& = & ~\sum \limits_{i=n}\limits^{\infty}\sum \limits_{j=1}\limits^{\infty}\mathbb{E}\left(\id_{A_{i}}\hat{g}_{j}(\x_{x,i},...,\x_{x,i+j-1},\beta_{i+1},...,\beta_{i+j})\hat{h}_{j}(\x_{x,i},...,\x_{x,i+j-1},\beta_{i+1},...,\beta_{i+j}) \right)\\
	& = & ~\sum \limits_{i=n}\limits^{\infty}\sum \limits_{j=1}\limits^{\infty}P(A_{i})\mathbb{E}\left(\hat{g}_{j}(\x_{\overline{x},0},...,\x_{\overline{x},j-1},\beta_{1},...,\beta_{j})\hat{h}_{j}(\x_{\overline{x},0},...,\x_{\overline{x},j-1},\beta_{1},...,\beta_{j}) \right)\\
	& = & ~P(A)\sum \limits_{j=1}\limits^{\infty}\mathbb{E}\left(\hat{g}_{j}(\x_{\overline{x},0},...,\x_{\overline{x},j-1},\beta_{1},...,\beta_{j})\hat{h}_{j}(\x_{\overline{x},0},...,\x_{\overline{x},j-1},\beta_{1},...,\beta_{j}) \right).
\end{eqnarray*}
In addition, it is straightforward that
\begin{align*}
\hat{g}_{j}(\x_{\overline{x},0},...,\x_{\overline{x},j-1},\beta_{1},...,\beta_{j})(\omega)=\id_{\{e_{\overline{x}}(1)=j\}}(\omega).
\end{align*}
Using this, while having in mind Lemma \ref{lemma_ineqx}.i), gives
\begin{eqnarray*}
	\mathbb{E}\left(\id_{A}f\left(\int \limits_{\alpha_{e_{x}(n)}}\limits^{\alpha_{e_{x}(n+1)}}\Xi(\X_{x}(\tau))d\tau\right) \right)
	& = & ~P(A) \sum \limits_{j=1}\limits^{\infty}\mathbb{E}\left(\id_{\{e_{\overline{x}}(1)=j\}}f\left( \sum \limits_{k=0}\limits^{j-1}\int \limits_{0}\limits^{\beta_{k+1}}\Xi(T_{\A}(\tau)\x_{\overline{x},k})d\tau\right) \right)\\
	& = & ~ P(A)\sum \limits_{j=1}\limits^{\infty}\mathbb{E}\left(\id_{\{e_{\overline{x}}(1)=j\}}f\left(\int \limits_{0}\limits^{\alpha_{e_{\overline{x}}(1)}}\Xi(\X_{\overline{x}}(\tau))d\tau\right) \right)\\
\end{eqnarray*}
Finally, as $e_{\overline{x}}(1)\in \mathbb{N}$ a.s. and as $A \in \F^{x}_{e_{x}(n)}$ was arbitrary, we obtain
\begin{align*}
\mathbb{E}\left(\id_{A}f\left(\int \limits_{\alpha_{e_{x}(n)}}\limits^{\alpha_{e_{x}(n+1)}}\Xi(\X_{x}(\tau))d\tau\right) \right) =P(A) \mathbb{E}\left(f\left(\int \limits_{0}\limits^{\alpha_{e_{\overline{x}}(1)}}\Xi(\X_{\overline{x}}(\tau))d\tau\right) \right),
\end{align*}
for all $A \in \F^{x}_{e_{x}(n)}$, which implies the claim, by the very definition of the conditional expectation.
\end{proof}

\begin{lemma}\label{lemma_measint} Let $(\Xi,(W,||\cdot||_{W}))\in SL_{V_{2}}(V)$, $n \in \mathbb{N}\setminus \{1\}$ and $x:\Omega \rightarrow V$ an independent initial leading to extinction. Then the mapping defined by
\begin{align*}
\Omega \ni \omega \mapsto \int \limits_{\alpha_{e_{x}(n-1)(\omega)}}\limits^{\alpha_{e_{x}(n)(\omega)}}\Xi(\X_{x}(\tau,\omega))d\tau,
\end{align*}
is $\F^{x}_{e_{x}(n)}$-$\B(W)$-measurable.
\end{lemma}
\begin{proof} As $(\F^{x}_{e_{x}(m)})_{m \in \mathbb{N}}$ is a filtration, it suffices to prove that  $\int \limits_{0}\limits^{\alpha_{e_{x}(n)}}\Xi(\X_{x}(\tau))d\tau$ is $\F^{x}_{e_{x}(n)}$-$\B(W)$-measurable, for all $n \in \mathbb{N}$. To this end, introduce $j \in \mathbb{N}$ as well as $B \in \B(W)$ and observe that
\begin{align}
\label{lemma_measintegralproofeq1}
\left\{\int\limits_{0}\limits^{\alpha_{e_{x}(n)}}\Xi(\X_{x}(\tau))d\tau \in B\right\} \cap\{e_{x}(n)=j\}= \left\{\sum \limits_{k=0}\limits^{j-1} \int\limits_{0}\limits^{\beta_{k+1}}\Xi(T_{\A}(\tau)\x_{x,k})d\tau \in B\right\}\cap\{e_{x}(n)=j\}.
\end{align}
As demonstrated in the proof of Lemma \ref{lemma_iidpreparation}, $(t,v)\mapsto \int \limits_{0}\limits^{t}\Xi(T_{\A}(\tau)v)\id_{V_{2}}(v)d\tau$ is $\B([0,\infty))\otimes \B(V)$-$\B(W)$-measurable. Consequently, since $\x_{x,k}$ and $\beta_{k+1}$ are $\F^{x}_{k+1}$-$\B(V)$-measurable and $\F^{x}_{k+1}$-$\B([0,\infty))$-measurable, resp., for all $k=0,...,j-1$, we get that
\begin{align*}
\sum \limits_{k=0}\limits^{j-1} \int\limits_{0}\limits^{\beta_{k+1}}\Xi(T_{\A}(\tau)\x_{x,k})d\tau = \sum \limits_{k=0}\limits^{j-1} \int\limits_{0}\limits^{\beta_{k+1}}\Xi(T_{\A}(\tau)\x_{x,k})\id_{V_{2}}(\x_{x,k})d\tau
\end{align*}
is  $\F^{x}_{j}$-$\B(W)$-measurable, where the equality holds almost surely. This gives, while having in mind (\ref{lemma_measintegralproofeq1}) as well as Lemma \ref{lemma_xbasicprop}.iii) that
\begin{align*}
\left\{\int\limits_{0}\limits^{\alpha_{e_{x}(n)}}\Xi(\X_{x}(\tau))d\tau \in B\right\} \cap\{e_{x}(n)=j\}\in \F^{x}_{j}
\end{align*}
and the claim follows.
\end{proof}

\begin{proposition}\label{prop_iid} Let $(\Xi,(W,||\cdot||_{W}))\in SL_{V_{2}}(V)$ and let $x:\Omega \rightarrow V$ be an independent initial leading to extinction. Then the sequence $\left(\int \limits_{\alpha_{e_{x}(n)}}\limits^{\alpha_{e_{x}(n+1)}}\Xi(\X_{x}(\tau))d\tau\right)_{n\in \mathbb{N}}$ is i.i.d., with
\begin{align}
\label{prop_iideq}
\int \limits_{\alpha_{e_{x}(n)}}\limits^{\alpha_{e_{x}(n+1)}}\Xi(\X_{x}(\tau))d\tau = \int \limits_{0}\limits^{\alpha_{e_{\overline{x}}(1)}}\Xi(\X_{\overline{x}}(\tau))d\tau
\end{align}
in distribution, for all $n \in \mathbb{N}$.
\end{proposition}
\begin{proof} Let $B \in \B(W)$ be given, and set $f:=\id_{B}$, where $f:W\rightarrow \mathbb{R}$. Then $f$ is obviously bounded and $\B(W)$-$\B(\R)$-measurable. Consequently, appealing to Lemma \ref{lemma_iidpreparation} yields
\begin{align*}
\P\left(\int \limits_{\alpha_{e_{x}(n)}}\limits^{\alpha_{e_{x}(n+1)}}\Xi(\X_{x}(\tau))d\tau \in B\right)= \mathbb{E}f\left(\int \limits_{\alpha_{e_{x}(n)}}\limits^{\alpha_{e_{x}(n+1)}}\Xi(\X_{x}(\tau))d\tau\right) = \mathbb{E}f\left(\int \limits_{0}\limits^{\alpha_{e_{\overline{x}}(1)}}\Xi(\X_{\overline{x}}(\tau))d\tau\right),
\end{align*}
which implies (\ref{prop_iideq}).\\
Consequently, it remains to show that
\begin{align}
\label{prop_iidproofeq1}
\P\left(\int \limits_{\alpha_{e_{x}(1)}}\limits^{\alpha_{e_{x}(2)}}\Xi(\X_{x}(\tau))d\tau \in B_{1},...,\int \limits_{\alpha_{e_{x}(n)}}\limits^{\alpha_{e_{x}(n+1)}}\Xi(\X_{x}(\tau))d\tau \in B_{n}\right) = \prod \limits_{k=1}\limits^{n}\P\left(\int \limits_{\alpha_{e_{x}(k)}}\limits^{\alpha_{e_{x}(k+1)}}\Xi(\X_{x}(\tau))d\tau \in B_{k}\right)
\end{align}
for all $B_{1},...,B_{n}\in \B(W)$ and $n \in \mathbb{N}$.\\
(\ref{prop_iidproofeq1}) is trivial if $n=1$. So assume it holds for $n-1\in \mathbb{N}$ and let us prove it for $n$. To this end, introduce $B_{1},...,B_{n}\in \B(W)$ and $f_{k}:=\id_{B_{k}}$. Then employing Lemma \ref{lemma_iidpreparation}, Lemma \ref{lemma_measint}, (\ref{prop_iidproofeq1}) and (\ref{prop_iideq}) yields
\begin{eqnarray*}
& & ~
	\P\left(\int \limits_{\alpha_{e_{x}(1)}}\limits^{\alpha_{e_{x}(2)}}\Xi(\X_{x}(\tau))d\tau \in B_{1},...,\int \limits_{\alpha_{e_{x}(n)}}\limits^{\alpha_{e_{x}(n+1)}}\Xi(\X_{x}(\tau))d\tau \in B_{n}\right)\\
	& = & ~ \mathbb{E}\left(\prod \limits_{k=1}\limits^{n-1}f_{k}\left(\int \limits_{\alpha_{e_{x}(k)}}\limits^{\alpha_{e_{x}(k+1)}}\Xi(\X_{x}(\tau))d\tau\right)\mathbb{E}\left(f_{n}\left(\int \limits_{\alpha_{e_{x}(n)}}\limits^{\alpha_{e_{x}(n+1)}}\Xi(\X_{x}(\tau))d\tau \right)\Big\vert\F_{e_{x}(n)}^{x}\right)\right) \\
	& = & ~ \prod \limits_{k=1}\limits^{n}\P\left(\int \limits_{\alpha_{e_{x}(k)}}\limits^{\alpha_{e_{x}(k+1)}}\Xi(\X_{x}(\tau))d\tau \in B_{k}\right) 
\end{eqnarray*}
and the claim follows.
\end{proof}

\begin{lemma}\label{lemma_squareintegrability} Let $(\Xi,(W,||\cdot||_{W}))\in SL_{V_{2}}(V)$ and let $x:\Omega \rightarrow V$ be an independent initial leading to extinction. Then, the assertion
\begin{align*}
\int \limits_{\alpha_{e_{x}(n)}}\limits^{\alpha_{e_{x}(n+1)}}\Xi(\X_{x}(\tau))d\tau\in L^{2}(\Omega;W)
\end{align*}
is valid for all $n \in \mathbb{N}$.
\end{lemma}
\begin{proof} The desired measurability follows a fortiori from Lemma \ref{lemma_measint}. Moreover, employing Proposition \ref{prop_iid} yields that it suffices to prove that
\begin{align*}
\mathbb{E}\left|\left|\int \limits_{0}\limits^{\alpha_{e_{\overline{x}}(1)}}\Xi(\X_{\overline{x}}(\tau))d\tau\right|\right|_{W}^{2}<\infty.
\end{align*}
To this end, note that 
\begin{align*}
\mathbb{E}\left|\left|\int \limits_{0}\limits^{\alpha_{e_{\overline{x}}(1)}}\Xi(\X_{\overline{x}}(\tau))d\tau\right|\right|_{W}^{2}\leq \mathbb{E}\left(\int \limits_{0}\limits^{\alpha_{e_{\overline{x}}(1)}}\left|\left|\Xi(\X_{\overline{x}}(\tau))\right|\right|_{W}d\tau\right)^{2}\leq  \mathbb{E}\left(\sum \limits_{k=0}\limits^{e_{\overline{x}}(1)-1}\beta_{k+1}\left(c_{1}||\x_{\overline{x},k}||_{V_{2}}+c_{2}\right)\right)^{2},
\end{align*}
where the second inequality follows from Lemma \ref{lemma_Xbasicprop}.i), Assumption \ref{assumption_fe}.iii) and Lemma \ref{lemma_ineqx}.i).\\
Now introduce $\eta_{0}:=\overline{x}$, for notational conveniences. Moreover, by the aid of Assumption \ref{assumption_fe}.iii) and Lemma \ref{lemma_ineqx}.i), it is easy to verify inductively that
\begin{align}
\label{lemma_squareintegrabilityproofeq1}
||\x_{\overline{x},k}||_{V_{2}}\leq \sum \limits_{j=0}\limits^{k}||\eta_{k}||_{V_{2}},~\forall k \in \mathbb{N}_{0}.
\end{align}
Consequently, we get
\begin{align*}
\mathbb{E}\left|\left|\int \limits_{0}\limits^{\alpha_{e_{\overline{x}}(1)}}\Xi(\X_{\overline{x}}(\tau))d\tau\right|\right|_{W}^{2} \leq \mathbb{E}\left(\sum \limits_{k=0}\limits^{e_{\overline{x}}(1)-1}\beta_{k+1}(c_{1}\sum \limits_{j=0}\limits^{k}||\eta_{k}||_{V_{2}}+c_{2})\right)^{2}.
\end{align*}
Hence, we also have
\begin{align*}
\mathbb{E}\left|\left|\int \limits_{0}\limits^{\alpha_{e_{\overline{x}}(1)}}\Xi(\X_{\overline{x}}(\tau))d\tau\right|\right|_{W}^{2} \leq \sum \limits_{m=1}\limits^{\infty} \mathbb{E}\left(\left(\sum \limits_{k=0}\limits^{m-1}\beta_{k+1}(c_{1}\sum \limits_{j=0}\limits^{k}||\eta_{k}||_{V_{2}}+c_{2})\right)^{2}\id_{\{e_{\overline{x}}(1)=m\}}\right).
\end{align*}
Consequently, appealing to Cauchy-Schwarz' inequality implies
\begin{align}
\label{lemma_squareintegrabilityproofeq2}
\mathbb{E}\left|\left|\int \limits_{0}\limits^{\alpha_{e_{\overline{x}}(1)}}\Xi(\X_{\overline{x}}(\tau))d\tau\right|\right|_{W}^{2} \leq \sum \limits_{m=1}\limits^{\infty} \left(\mathbb{E}\left(\sum \limits_{k=0}\limits^{m-1}\beta_{k+1}(c_{1}\sum \limits_{j=0}\limits^{k}||\eta_{k}||_{V_{2}}+c_{2})\right)^{4}\right)^{\frac{1}{2}}\P(e_{\overline{x}}(1)=m)^{\frac{1}{2}}.
\end{align}
Now upper bounds for each factor of each summand of the preceding series will be derived.\\
So let $m \in \mathbb{N}$ be arbitrary but fixed. Then the triangle inequality, the independence of $(\beta_{k})_{k \in \mathbb{N}}$ and $(\eta_{m})_{m \in \mathbb{N}}$ as well as the fact that each of these sequences is identically distributed, yields
\begin{eqnarray*}
\left(\mathbb{E}\left(\sum \limits_{k=0}\limits^{m-1}\beta_{k+1}(c_{1}\sum \limits_{j=0}\limits^{k}||\eta_{k}||_{V_{2}}+c_{2})\right)^{4}\right)^{\frac{1}{4}}
& \leq & ~\sum \limits_{k=0}\limits^{m-1}||\beta_{k+1}||_{L^{4}(\Omega)}\left(c_{1}\sum \limits_{j=0}\limits^{k}||~||\eta_{k}||_{V_{2}}||_{L^{4}(\Omega)}+c_{2}\right)\\
& = & ~ ||\beta_{1}||_{L^{4}(\Omega)}c_{1}||~||\eta_{1}||_{V_{2}}||_{L^{4}(\Omega)}\frac{m(m+1)}{2}+||\beta_{1}||_{L^{4}(\Omega)}c_{2}m\\
& \leq & ~ m^{2}\left(||\beta_{1}||_{L^{4}(\Omega)}c_{1}||~||\eta_{1}||_{V_{2}}||_{L^{4}(\Omega)}+||\beta_{1}||_{L^{4}(\Omega)}c_{2}\right). 
\end{eqnarray*}
Note that $||\beta_{1}||_{L^{4}(\Omega)}<\infty$ and $||~||\eta_{1}||_{V_{2}}||_{L^{4}(\Omega)}<\infty$, by Assumption \ref{assumption_mb}.ii).\\
Consequently, by introducing $C:=\left(||\beta_{1}||_{L^{4}(\Omega)}c_{1}||~||\eta_{1}||_{V_{2}}||_{L^{4}(\Omega)}+||\beta_{1}||_{L^{4}(\Omega)}c_{2}\right)^{2}<\infty$, we get
\begin{align}
\label{lemma_squareintegrabilityproofeq3}
\left(\mathbb{E}\left(\sum \limits_{k=0}\limits^{m-1}\beta_{k+1}(c_{1}\sum \limits_{j=0}\limits^{k}||\eta_{k}||_{V_{2}}+c_{2})\right)^{4}\right)^{\frac{1}{2}} \leq Cm^{4},~\forall m \in \mathbb{N}.
\end{align}
Now for all $m \in \mathbb{N}\setminus \{1\}$ we have
\begin{align*}
\P(e_{\overline{x}}(1)=m)\leq \P(T(\beta_{k})\x_{\overline{x},k-1}\neq 0,~\forall k=1,...,m-1).
\end{align*}
Consequently, employing Assumption \ref{assumption_fe}.ii), which is possible due to Lemma \ref{lemma_ineqx}.i), yields
\begin{align*}
\P(e_{\overline{x}}(1)=m) \leq \P(-\kappa \beta_{k}+||\x_{\overline{x},k-1}||_{V_{1}}^{\rho}> 0,~\forall k=1,...,m-1)
\end{align*}
Hence by appealing to Lemma \ref{lemma_inclusionx} we get
\begin{align*}
\P(e_{\overline{x}}(1)=m) \leq \P\left(-\kappa \sum \limits_{k=1}\limits^{m-1}\beta_{k}+\sum \limits_{k=1}\limits^{m-2}||\eta_{k}||_{V_{1}}^{\rho}+||\eta_{0}||_{V_{1}}^{\rho}>0\right) = \P\left(\sum \limits_{k=1}\limits^{m-1}-\kappa\beta_{k}+||\eta_{k-1}||_{V_{1}}^{\rho}>0\right),
\end{align*}
for all $m \in \mathbb{N}\setminus \{1\}$. Now let $\nu:=\mathbb{E}(-\kappa\beta_{1}+||\eta_{0}||_{V_{1}}^{\rho})$, which is negative by Assumption \ref{assumption_mb}.iii). Consequently, we have
\begin{align}
\label{lemma_squareintegrabilityproofeq4}
\P(e_{\overline{x}}(1)=m) \leq \P\left(\left|\sum \limits_{k=1}\limits^{m-1}-\kappa\beta_{k}+||\eta_{k-1}||_{V_{1}}^{\rho}-\nu(m-1)\right|>|\nu|(m-1)\right)
\end{align}
for all $m \in \mathbb{N}\setminus \{1\}$. Hence, combining (\ref{lemma_squareintegrabilityproofeq2}), (\ref{lemma_squareintegrabilityproofeq3}) and (\ref{lemma_squareintegrabilityproofeq4}) yields
\begin{align*}
\mathbb{E}\left|\left|\int \limits_{0}\limits^{\alpha_{e_{\overline{x}}(1)}}\Xi(\X_{\overline{x}}(\tau))d\tau\right|\right|_{W}^{2} \leq C+ \sum \limits_{m=2}\limits^{\infty} Cm^{4}\P\left(\left|\sum \limits_{k=1}\limits^{m-1}-\kappa\beta_{k}+||\eta_{k-1}||_{V_{1}}^{\rho}-\nu(m-1)\right|>|\nu|(m-1)\right)^{\frac{1}{2}}
\end{align*}
Moreover, it is plain that $m \leq 2(m-1)$ for all $m \geq 2$ and consequently $m^{4}\leq 16(m-1)^{4}$, which yields by employing Cauchy Schwarz' inequality that
\begin{eqnarray*}
	& & ~
\mathbb{E}\left|\left|\int \limits_{0}\limits^{\alpha_{e_{\overline{x}}(1)}}\Xi(\X_{\overline{x}}(\tau))d\tau\right|\right|_{W}^{2}\\
& \leq & ~ C+ 16C\sum \limits_{m=1}\limits^{\infty} m^{4}\P\left(\left|\sum \limits_{k=1}\limits^{m}-\kappa\beta_{k}+||\eta_{k-1}||_{V_{1}}^{\rho}-\nu m\right|>|\nu|m\right)^{\frac{1}{2}}\\
& \leq & ~ C+ 16C\left(\sum \limits_{m=1}\limits^{\infty}m^{-1-\hat{\varepsilon}}\right)^{\frac{1}{2}}\left(\sum \limits_{m=1}\limits^{\infty} m^{9+\hat{\varepsilon}}\P\left(\left|\sum \limits_{k=1}\limits^{m}-\kappa\beta_{k}+||\eta_{k-1}||_{V_{1}}^{\rho}-\nu m\right|>|\nu|m\right)\right)^{\frac{1}{2}}.
\end{eqnarray*}
It is common knowledge that the first series in the preceding expression is finite. Consequently, the claim follows if the second is finite as well. To this end, note that the sequence $(-\kappa\beta_{k}+||\eta_{k-1}||_{V_{1}}^{\rho})_{k \in \mathbb{N}}$ is i.i.d. with mean $\nu$. Consequently, \cite[Theorem 1]{katz} yields
\begin{align*}
\sum \limits_{m=1}\limits^{\infty} m^{9+\hat{\varepsilon}}\P\left(\left|\sum \limits_{k=1}\limits^{m}-\kappa\beta_{k}+||\eta_{k-1}||_{V_{1}}^{\rho}-\nu m\right|>|\nu|m\right) < \infty,
\end{align*}
if (and only if) $-\kappa \beta_{1}+||\eta_{0}||_{V_{1}}^{\rho} \in L^{11+\hat{\varepsilon}}(\Omega)$, which is true by Assumption \ref{assumption_mb}.ii).
\end{proof}

Note that $(\varphi,\mathbb{R})\in SL_{V_{2}}(V)$, where $\varphi:V\rightarrow\mathbb{R}$ is the function which is constantly one. This plain fact, together with Proposition \ref{prop_iid} and Lemma \ref{lemma_squareintegrability} yields the following quite useful corollary.

\begin{corollary}\label{lemma_alphaiid} Let $x:\Omega \rightarrow V$ be an independent initial leading to extinction. Then the sequence $(\alpha_{e_{x}(n+1)}-\alpha_{e_{x}(n)})_{n \in \mathbb{N}}$ is square integrable and i.i.d with $\alpha_{e_{x}(n+1)}-\alpha_{e_{x}(n)}=\alpha_{e_{\overline{x}}(1)}$ in distribution.
\end{corollary}

\begin{lemma}\label{lemma_convcomp} Let $(U,||\cdot||_{U})$ be a separable Banach space. Moreover, let $(Y_{m})_{m \in \mathbb{N}}\subseteq \mathcal{M}(\Omega;U)$ be such that there is a $Y \in \mathcal{M}(\Omega;U)$, with $\lim \limits_{m \rightarrow \infty} Y_{m}=Y$ almost surely. Finally, let $(N_{t})_{t\geq 0}$, with \linebreak$N_{t}:\Omega \rightarrow \mathbb{N}$, be such that each $N_{t}$ is $\F$-$2^{\mathbb{N}}$-measurable and $\lim \limits_{t\rightarrow \infty}N_{t}=\infty$ almost surely. Then the convergence $\lim \limits_{t\rightarrow \infty} Y_{N_{t}}=Y$ takes place with probability one.
\end{lemma}
\begin{proof} Let $M\in \F$ be a $\P$-null-set such that $\lim \limits_{m \rightarrow \infty} Y_{m}(\omega)=Y(\omega)$ and $\lim \limits_{t\rightarrow \infty}N_{t}(\omega)=\infty$ for all \linebreak$\omega \in \Omega\setminus M$. Now fix one of these $\omega \in \Omega\setminus M$ and note that there is for each $\varepsilon>0$ an $m_{0}\in \mathbb{N}$ such that $||Y_{m}(\omega)-Y(\omega)||_{U}<\varepsilon$ for all $m \geq m_{0}$. In addition, we can find a $t_{0}\in [0,\infty)$ such that $N_{t}(\omega)\geq m_{0}$ for all $t \in [t_{0},\infty)$. Consequently, we get $||Y_{N_{t}(\omega)}(\omega)-Y(\omega)||_{U}<\varepsilon$ for all $t \geq t_{0}$, which yields the claim.
\end{proof}

Actually, the preceding result already enables us to prove our SLLN. But, to also prove our CLT, a version of Anscombe's CLT in type $2$ Banach spaces is needed. Since the standard CLT as well as Kolmogorov's inequality both hold if (and only if) the underlying Banach space is of type $2$, it is possible to prove our type $2$ version of Anscombe's theorem identically to the real-valued case. Since this is not that obvious from the statement of the theorem, the proof will be given. For a proof of Anscombe's theorem on the line see \cite[Theorem 3.2]{good}.

\begin{theorem}\label{theorem_anscombeclt} Let $(W,||\cdot||_{W})$ be a separable Banach space of type $2$, introduce $(Y_{m})_{m\in \mathbb{N}}\subseteq L^{2}(\Omega;W)$, $(N_{n})_{n\in \mathbb{N}}$, where $N_{n}:\Omega \rightarrow \mathbb{N}$ is $\F$-$2^{\mathbb{N}}$-measurable and $(\theta_{n})_{n\in \mathbb{N}}\subseteq (0,\infty)$. Moreover, assume that
\begin{enumerate}
	\item $(Y_{m})_{m \in \mathbb{N}}$ is i.i.d. and $\mathbb{E}Y_{1}=0$ and
	\item $\lim \limits_{n  \rightarrow \infty} \theta_{n}=\infty$ and  $\lim \limits_{n  \rightarrow \infty} \frac{N_{n}}{\theta_{n}}=1$ in probability.
\end{enumerate}
Then the convergence
\begin{align}
\label{theorem_ascombeclteq1}
\lim \limits_{n  \rightarrow \infty} \frac{1}{\sqrt{\theta_{n}}}\sum \limits_{k=1}\limits^{N_{n}}Y_{k}= \lim \limits_{n  \rightarrow \infty} \frac{1}{\sqrt{N_{n}}}\sum \limits_{k=1}\limits^{N_{n}}Y_{k}=Z,
\end{align}
takes place in distribution, where $Z \sim N_{W}(0,\Cov_{W}(Y_{1}))$.
\end{theorem}
\begin{proof} Firstly, the claim is trivial if $Y_{1}=0$ a.s., so assume w.l.o.g. $Y_{1}\neq 0$. Now, introduce $S_{n}:=\sum \limits_{k=1}\limits^{n}Y_{k}$, $\hat{S}_{n}:=\frac{1}{\sqrt{n}}S_{n}$ for all $n \in \mathbb{N}$ and let us start by proving the second equality in (\ref{theorem_ascombeclteq1}). Appealing to the CLT in type 2 Banach spaces, see \cite[Corollary 3.3 and Remark 1.1]{CLT}, yields $\lim \limits_{n \rightarrow \infty}\hat{S}_{n}=Z$ in distribution. Now set $\tilde{\theta}_{n}:=\lceil\theta_{n}\rceil$ and note that clearly $\lim \limits_{n \rightarrow \infty}\hat{S}_{\tilde{\theta}_{n}}=Z$ in distribution and $\lim \limits_{n  \rightarrow \infty} \frac{N_{n}}{\tilde{\theta}_{n}}=1$ in probability. Moreover, as $\hat{S}_{N_n}=(\hat{S}_{\tilde{\theta}_{n}}+\frac{S_{N_n}-S_{\tilde{\theta}_{n}}}{\sqrt{\tilde{\theta}_{n}}})\sqrt{\frac{\tilde{\theta}_{n}}{N_n}}$ for all $n \in \mathbb{N}$, Slutsky's theorem\footnote{We were unable to find a direct reference for Slutsky's theorem in the Banach space setting. However, this is easily deduced from \cite[Theorem 3.9]{Billingsley} and the continuous mapping theorem.} yields that the second equality in (\ref{theorem_ascombeclteq1}) follows, if
\begin{align}
\label{theorem_anscombeproofeq1}
\lim \limits_{n \rightarrow \infty}\frac{S_{N_n}-S_{\tilde{\theta}_{n}}}{\sqrt{\tilde{\theta}_{n}}}=0,
\end{align}
in probability.\\
So let us prove (\ref{theorem_anscombeproofeq1}). To this end, let $\varepsilon>0$, $\delta \in (0,1)$, $r_{n}:=\lceil \tilde{\theta}_{n}(1-\delta)\rceil$ and $R_{n}:=\lfloor \tilde{\theta}_{n}(1+\delta)\rfloor$ for all $n \in \mathbb{N}$. And note that it is plain that
\begin{align*}
\P\left(||S_{N_{n}}-S_{\tilde{\theta}_n}||_{W}>\varepsilon\sqrt{\tilde{\theta}_{n}}\right)\leq \P\left(||S_{N_{n}}-S_{\tilde{\theta}_n}||_{W}>\varepsilon\sqrt{\tilde{\theta}_{n}},~ \left|\frac{N_{n}}{\tilde{\theta}_n}-1\right|\leq \delta\right)+\P\left(\left|\frac{N_{n}}{\tilde{\theta}_n}-1\right|> \delta\right)
\end{align*}
for all $n \in \mathbb{N}$. Moreover, as $\left|\frac{N_{n}}{\tilde{\theta}_n}-1\right|\leq \delta$ if and only if $N_{n}\in [r_{n},R_{n}]$, we get
\begin{align*}
\P\left(||S_{N_{n}}-S_{\tilde{\theta}_n}||_{W}>\varepsilon\sqrt{\tilde{\theta}_{n}},~ \left|\frac{N_{n}}{\tilde{\theta}_n}-1\right|\leq \delta\right) \leq \P\left(\max \limits_{m=r_n,..,R_n}||S_{m}-S_{\tilde{\theta}_n}||_{W}>\varepsilon\sqrt{\tilde{\theta}_{n}}\right),
\end{align*}
for all $n \in \mathbb{N}$. In addition, note that
\begin{eqnarray*}
\max \limits_{m=r_n,..,R_n}||S_{m}-S_{\tilde{\theta}_n}||_{W}
& \leq & ~ \max \limits_{m=r_n,..,\tilde{\theta}_n-1}\left|\left|\sum \limits_{k=m+1}\limits^{\tilde{\theta}_n}Y_{k}\right|\right|_{W}+\max \limits_{m=\tilde{\theta}_n+1,..,R_n}\left|\left|\sum \limits_{k=\tilde{\theta}_n+1}\limits^{m}Y_{k}\right|\right|_{W}\\
& = & ~ \max \limits_{m=1,..,\tilde{\theta}_n-r_n}\left|\left|\sum \limits_{k=1}\limits^{m}Y_{\tilde{\theta}_{n}+1-k}\right|\right|_{W}+\max \limits_{m=1,..,R_n-\tilde{\theta}_n}\left|\left|\sum \limits_{k=1}\limits^{m}Y_{k+\tilde{\theta}_n}\right|\right|_{W}
\end{eqnarray*}
where we set $ \max \limits_{m=a,..,b}(\cdot):=0$, if $a>b$.\\ 
Using this, together with the well known inequality $P(X_{1}+X_{2}>t)\leq \P(2X_{1}> t)+\P(2X_{2}> t)$, for any $X_{1},X_{2}\in\mathcal{M}(\Omega;\mathbb{R})$, $t > 0$ and Kolmogorov's inequality in type $2$ Banach spaces (see \cite[Theorem 6.1]{CLT}), yields that there is a constant $C>0$ such that
\begin{align*}
\P\left(\max \limits_{m=r_n,..,R_n}||S_{m}-S_{\tilde{\theta}_n}||_{W}>\varepsilon\sqrt{\tilde{\theta}_{n}}\right)\leq \frac{4C\mathbb{E}||Y_{1}||_{W}^{2}}{\varepsilon^{2}{\tilde{\theta}_n}}(\tilde{\theta}_{n}-r_{n})+\frac{4C\mathbb{E}||Y_{1}||_{W}^{2}}{\varepsilon^{2}{\tilde{\theta}_n}}(R_{n}-\tilde{\theta}_{n}),
\end{align*}
for all $n \in \mathbb{N}$. Now let $\varepsilon^{\prime}>0$ be arbitrary but fixed and choose $0<\delta<\min\left(\frac{\varepsilon^{2}\varepsilon^{\prime}}{8C\mathbb{E}||Y_{1}||_{W}^{2}},1\right)$, then we get
\begin{align*}
\P\left(\max \limits_{m=r_n,..,R_n}||S_{m}-S_{\tilde{\theta}_n}||_{W}>\varepsilon\sqrt{\tilde{\theta}_{n}}\right) \leq \frac{4C\mathbb{E}||Y_{1}||_{W}^{2}}{\varepsilon^{2}{\tilde{\theta}_n}}(R_n-r_n) \leq \frac{8C\mathbb{E}||Y_{1}||_{W}^{2}}{\varepsilon^{2}}\delta \leq \varepsilon^{\prime}.
\end{align*}
Conclusively, putting it all together yields $\limsup \limits_{n \rightarrow \infty } \P\left(||S_{N_{n}}-S_{\tilde{\theta}_n}||_{W}>\varepsilon\sqrt{\tilde{\theta}_{n}}\right) \leq \varepsilon^{\prime}$, which implies (\ref{theorem_anscombeproofeq1}), since $\varepsilon^{\prime}>0$ can be chosen arbitrarily small. Finally, the first inequality in (\ref{theorem_ascombeclteq1}) now follows from the second one an Slutsky's theorem.
\end{proof}

\begin{theorem}\label{theorem_sllncltmain} Let $(\Xi,(W,||\cdot||_{W}))\in SL_{V_{2}}(V)$ and let $x:\Omega \rightarrow V$ be an independent initial leading to extinction. Moreover, introduce $\nu:=\frac{1}{\mathbb{E}(\alpha_{e_{\overline{x}}(1)})}\mathbb{E}\left(\int \limits_{0}\limits^{\alpha_{e_{\overline{x}}(1)}}\Xi (\X_{\overline{x}}(\tau))d\tau\right)$. Then the convergence
\begin{align}
\label{theorem_sllnclteq1}
\lim \limits_{t\rightarrow \infty} \frac{1}{t}\int \limits_{0}\limits^{t}\Xi(\X_{x}(\tau))d\tau= \nu,
\end{align}
takes place almost surely in $(W,||\cdot||_{W})$\footnote{This of course means convergence for $\P$-a.e. $\omega \in \Omega$ with respect to $||\cdot||_{W}$. So far it seems redundant to write "almost surely w.r.t. $||\cdot||_{W}$", instead of just "almost surely". But later on we will choose $W$ as a subspace of $V$, which makes it necessary to emphasize w.r.t. which norm the almost sure convergence is taking place.}. Moreover, if $(W,||\cdot||_{W})$ is of type $2$, then
\begin{align}
\label{theorem_sllnclteq2}
\lim \limits_{t  \rightarrow \infty}\frac{1}{\sqrt{t}}\left(\int \limits_{0}\limits^{t}\Xi(\X_{x}(\tau))d\tau-t\nu\right)=Z,
\end{align}
in distribution, as elements of $(W,||\cdot||_{W})$\footnote{Again, in the next theorem it becomes clear why we emphasize on the fact that these are elements of $(W,||\cdot||_{W})$.}, where $Z \sim N_{W}(0,Q)$ and the covariance is given by\linebreak $Q:=\Cov_{W}\left(\sqrt{\frac{1}{\mathbb{E}(\alpha_{e_{\overline{x}}(1)})}}\int \limits_{0}\limits^{\alpha_{e_{\overline{x}}(1)}}\Xi (\X_{\overline{x}}(\tau))-\nu d\tau\right)$.
\end{theorem}
\begin{proof} Until explicitly stated otherwise, $(W,||\cdot||_{W})$ is not necessarily of type $2$.\\ 
Firstly, note that both expectations occurring in the definition of $\nu$ are indeed finite by Proposition \ref{prop_iid}, Lemma \ref{lemma_squareintegrability} and Corollary \ref{lemma_alphaiid}. Now, introduce $\Xi_{\nu}:V\rightarrow W$, by $\Xi_{\nu}(v):=\Xi(v)-\nu$ for all $v \in V$; and $(Y_{k})_{k \in \mathbb{N}_{0}}$, with $Y_{k}:\Omega \rightarrow W$ for all $k \in \mathbb{N}_{0}$, by $Y_{k}:=\int \limits_{\alpha_{e_{x}(k)}}\limits^{\alpha_{e_{x}(k+1)}}\Xi_{\nu}(\X_{x}(\tau))d\tau$ for all $k \in \mathbb{N}$ and $Y_{0}:=\int \limits_{0}\limits^{\alpha_{e_{\overline{x}}(1)}}\Xi_{\nu}(\X_{\overline{x}}(\tau))d\tau$. Finally, let $L(t):\Omega \rightarrow \mathbb{N}_{0}$ be defined by $L(t):=\max(k \in \mathbb{N}_{0}:~\alpha_{e_{x}(k)}\leq t)$ for all $t\geq 0$, where $e_{x}(0):=0$\\
Now we will proceed by proving the following assertions, from which (\ref{theorem_sllnclteq1}) as well as (\ref{theorem_sllnclteq2}) will follow quickly.
\begin{enumerate}
	\item $\mathbb{E}\alpha_{e_{\overline{x}}(1)}>0$ and $\lim \limits_{t  \rightarrow \infty}\frac{L(t)+1}{t}=\frac{1}{\mathbb{E}\alpha_{e_{\overline{x}}(1)}}$ almost surely.
	\item $\Xi_{\nu}\in SL_{V_{2}}(V)$ and consequently $(Y_{m})_{m\in \mathbb{N}}\subseteq L^{2}(\Omega;W)$ is centered, i.i.d. and $Y_{m}=Y_{0}$ in distribution for all $m \in \mathbb{N}$.
	\item $\lim \limits_{t  \rightarrow \infty}\frac{1}{\sqrt{t}}\left(\int \limits_{0}\limits^{t}\Xi_{\nu}(\X_{x}(\tau))d\tau-\sum \limits_{k=1}\limits^{L(t)+1}Y_{k}\right)=0$ almost surely.
\end{enumerate}
Proof of i). Firstly, note that $\P(L(t)<\infty,~\forall t \geq 0)=1$, since: Employing Corollary \ref{lemma_alphaiid} and the SLLN yields
\begin{align}
\label{theorem_sllncltproofeq1}
\lim \limits_{k \rightarrow \infty}\frac{1}{k} \alpha_{e_{x}(k)}=\lim \limits_{k \rightarrow \infty}\frac{1}{k}\alpha_{e_{x}(1)}+\frac{k-1}{k}\frac{1}{k-1} \sum \limits_{j=1}\limits^{k-1} (\alpha_{e_{x}(j+1)}-\alpha_{e_{x}(j)}) = \mathbb{E}\alpha_{e_{\overline{x}}(1)}> 0,
\end{align}
almost surely, where the last inequality follows from $\alpha_{e_{\overline{x}}(1)}\geq \alpha_{1}>0$ almost surely. Consequently, if there were a $t\geq 0$ such that $\P(L(t)=\infty)>0$, then
\begin{align*}
0<\P(L(t)=\infty)=\P(\alpha_{e_{x}(k)}\leq t,~\forall k \in \mathbb{N})\leq\P(\lim \limits_{k \rightarrow\infty}\frac{1}{k} \alpha_{e_{x}(k)}- \frac{t}{k}\leq 0)=0.
\end{align*}
Hence, $\P(L(t)<\infty)=1$ for a given $t$, which yields $\P(L(t)<\infty,~\forall t \geq 0)=1$, as the paths of $L(t)$ are clearly increasing with probability one.\\ 
Moreover, it is plain to verify that the simple but quite useful inequality
\begin{align}
\label{theorem_sllncltproofeq0}
\alpha_{e_{x}(L(t))}\leq t \leq \alpha_{e_{x}(L(t)+1)},~\forall t\geq 0
\end{align}
takes place with probability one. Particularly, we have
\begin{align*}
\frac{\alpha_{e_{x}(L(t))}}{L(t)+1}\leq \frac{t}{L(t)+1} \leq \frac{\alpha_{e_{x}(L(t)+1)}}{L(t)+1}
\end{align*}
for all $t\geq 0$, almost surely. Furthermore, thanks to (\ref{theorem_sllncltproofeq1}), it is plain that also $\lim \limits_{k \rightarrow \infty}\frac{1}{k} \alpha_{e_{x}(k-1)}=\mathbb{E}\alpha_{e_{\overline{x}}(1)}$ almost surely. Consequently, if $\lim\limits_{t  \rightarrow \infty} L(t)+1=\infty$ a.s., then employing (\ref{theorem_sllncltproofeq1}), Lemma \ref{lemma_convcomp}, the previous inequality as well as the sandwich lemma give i). Hence, i) follows once $\lim\limits_{t  \rightarrow \infty}L(t)=\infty$ a.s. is proven.\\
To this end, let $M  \in \F$ be a $\P$-null-set, such that $\alpha_{e_{x}(k)}(\omega)$ is well-defined for all $k \in \mathbb{N}_{0}$ and such that $\lim \limits_{k \rightarrow \infty}\frac{1}{k} \alpha_{e_{x}(k)}(\omega)=\mathbb{E}\alpha_{e_{\overline{x}}(1)}$, for all $\omega \in \Omega \setminus M$. Now fix one these $\omega$ and note that there is for a given $\varepsilon>0$ a $k_{0}\in \mathbb{N}$, such that $\left|\frac{1}{k} \alpha_{e_{x}(k)}(\omega)-\mathbb{E}\alpha_{e_{\overline{x}}(1)}\right|<\varepsilon$ for all $k \geq k_{0}$. Hence, choosing $\varepsilon=\mathbb{E}\alpha_{e_{\overline{x}}(1)}$ yields the existence of a $k_{0}\in \mathbb{N}$, with
$0<\alpha_{e_{x}(k)}(\omega)<2k\mathbb{E}\alpha_{e_{\overline{x}}(1)}$ for all $k \geq k_{0}$, and hence
\begin{align*}
\sup \limits_{t  \geq 0} L(t)(\omega) \geq \sup \limits_{k \geq k_{0}} L(2k\mathbb{E}\alpha_{e_{\overline{x}}(1)})(\omega)\geq  \sup \limits_{k \geq k_{0}} k =\infty.
\end{align*}
Finally, this implies $\lim \limits_{t  \rightarrow \infty } L(t)=\infty$ a.s., since $M$ is a $\P$-null-set and $L$ has paths that increase with probability one.\\
Proof of ii). Employing Remark \ref{lemma_xichoice}.ii) yields that $\Xi_{\nu}\in SL_{V_{2}}(V)$. Consequently, appealing to Lemma \ref{lemma_squareintegrability} as well as Proposition \ref{prop_iid} yields all claims in ii), except for $\mathbb{E}Y_{k}=0$ for all $k \in \mathbb{N}_{0}$. But this is plain, since $Y_{k}=Y_{0}$ in distribution gives
\begin{align*}
\mathbb{E}Y_{k}=\mathbb{E}Y_{0}=\mathbb{E}\left(\int \limits_{0}\limits^{\alpha_{e_{\overline{x}}(1)}}\Xi_{\nu}(\X_{\overline{x}}(\tau))d\tau\right) = \mathbb{E}\left(\int \limits_{0}\limits^{\alpha_{e_{\overline{x}}(1)}}\Xi(\X_{\overline{x}}(\tau))d\tau\right)-\nu \mathbb{E}\alpha_{e_{\overline{x}}(1)}=0,
\end{align*}
for all $k \in \mathbb{N}_{0}$.\\
Proof of iii). Let us start by proving that
\begin{align}
\label{theorem_sllncltproofeq2}
\lim \limits_{t  \rightarrow \infty}\frac{1}{\sqrt{t}}\int \limits_{\alpha_{e_{x}(L(t))}}\limits^{\alpha_{e_{x}(L(t)+2)}}||\Xi_{\nu}(\X_{x}(\tau))||_{W}d\tau=0
\end{align}
with probability one. Firstly, ii) and Remark \ref{lemma_xichoice}.i) yield $(||\Xi_{v}||_{W},\mathbb{R})\in SL_{V_{2}}(V)$. Consequently, invoking Lemma \ref{lemma_squareintegrability} and Proposition \ref{prop_iid} yields that $\left(\left(\int \limits_{\alpha_{e_{x}(n)}}\limits^{\alpha_{e_{x}(n+1)}}||\Xi_{\nu}(\X_{x}(\tau))||_{W}d\tau\right)^{2}\right)_{n\in \mathbb{N}}$ is integrable and i.i.d. Hence by appealing to the SLLN we get
\begin{eqnarray*}
	& & ~
\lim \limits_{n  \rightarrow \infty} \frac{1}{n}\left(\int \limits_{\alpha_{e_{x}(n)}}\limits^{\alpha_{e_{x}(n+1)}}||\Xi_{\nu}(\X_{x}(\tau))||_{W}d\tau\right)^{2}\\
& = & ~ \lim \limits_{n  \rightarrow \infty} \frac{1}{n}\sum\limits_{k=1}\limits^{n}\left(\int \limits_{\alpha_{e_{x}(k)}}\limits^{\alpha_{e_{x}(k+1)}}||\Xi_{\nu}(\X_{x}(\tau))||_{W}d\tau\right)^{2}-\frac{n-1}{n}\frac{1}{n-1}\sum\limits_{k=1}\limits^{n-1}\left(\int \limits_{\alpha_{e_{x}(k)}}\limits^{\alpha_{e_{x}(k+1)}}||\Xi_{\nu}(\X_{x}(\tau))||_{W}d\tau\right)^{2}\\
& = & ~ 0
\end{eqnarray*}
almost surely. Consequently, we also get
\begin{eqnarray*}
	& & ~
\lim \limits_{n  \rightarrow \infty} \frac{1}{\sqrt{n}}\int \limits_{\alpha_{e_{x}(n-1)}}\limits^{\alpha_{e_{x}(n+1)}}||\Xi_{\nu}(\X_{x}(\tau))||_{W}d\tau\\
& = & ~ \lim \limits_{n  \rightarrow \infty} \sqrt{\frac{n-1}{n}}\frac{1}{\sqrt{n-1}}\int \limits_{\alpha_{e_{x}(n-1)}}\limits^{\alpha_{e_{x}(n)}}||\Xi_{\nu}(\X_{x}(\tau))||_{W}d\tau+\frac{1}{\sqrt{n}}\int \limits_{\alpha_{e_{x}(n)}}\limits^{\alpha_{e_{x}(n+1)}}||\Xi_{\nu}(\X_{x}(\tau))||_{W}d\tau\\
& = & ~ 0.
\end{eqnarray*}
almost surely. In addition, i) enables us to apply Lemma \ref{lemma_convcomp} to the preceding equality, which gives
\begin{align*}
\lim \limits_{t  \rightarrow \infty} \frac{1}{\sqrt{L(t)+1}}\int \limits_{\alpha_{e_{x}(L(t))}}\limits^{\alpha_{e_{x}(L(t)+2)}}||\Xi_{\nu}(\X_{x}(\tau))||_{W}d\tau=0
\end{align*}
almost surely; this yields (\ref{theorem_sllncltproofeq2}) by employing i) once more. Finally, appealing to (\ref{theorem_sllncltproofeq0}), while having in mind (\ref{theorem_sllncltproofeq2}), yields
\begin{eqnarray*}
	& & ~
	\lim \limits_{t  \rightarrow \infty}\frac{1}{\sqrt{t}}\left|\left|\int \limits_{0}\limits^{t}\Xi_{\nu}(\X_{x}(\tau))d\tau-\sum \limits_{k=1}\limits^{L(t)+1}Y_{k}\right|\right|_{W}\\
	& \leq & ~ \lim \limits_{t  \rightarrow \infty}\frac{1}{\sqrt{t}}\left|\left|\int \limits_{0}\limits^{t}\Xi_{\nu}(\X_{x}(\tau))d\tau-\int \limits_{ 0}\limits^{ \alpha_{e_{x}(L(t)+2)}}\Xi_{\nu}(\X_{x}(\tau))d\tau\right|\right|_{W}+ \frac{1}{\sqrt{t}}\left|\left|\int \limits_{ 0}\limits^{ \alpha_{e_{x}(1)}}\Xi_{\nu}(\X_{x}(\tau))d\tau\right|\right|_{W}\\
	& \leq & ~ \lim \limits_{t  \rightarrow \infty}\frac{1}{\sqrt{t}}\int \limits_{ \alpha_{e_{x}(L(t))}}\limits^{ \alpha_{e_{x}(L(t)+2)}}\left|\left|\Xi_{\nu}(\X_{x}(\tau))\right|\right|_{W}d\tau\\
	& = & ~ 0,
\end{eqnarray*}
with probability one.\\
Now (\ref{theorem_sllnclteq1}) will be proven. Firstly, ii) and the SLLN in separable Banach spaces, see \cite[Corollary  7.10]{PIBS}, enable us to conclude that $\lim \limits_{n  \rightarrow \infty} \frac{1}{n}\sum \limits_{k=1}\limits^{n}Y_{k}=0$ a.s. Using this, as well as Lemma \ref{lemma_convcomp} and i) gives
\begin{align*}
\lim \limits_{t  \rightarrow \infty} \frac{1}{t} \sum \limits_{k=1}\limits^{L(t)+1}Y_{k}=\lim \limits_{t  \rightarrow \infty} \frac{L(t)+1}{t} \frac{1}{L(t)+1}\sum \limits_{k=1}\limits^{L(t)+1}Y_{k}=0,
\end{align*}
with probability one. Conclusively, Appealing to the previous equality, while having in mind iii), implies
\begin{eqnarray*}
	& & ~
	\lim \limits_{t  \rightarrow \infty }\left|\left|\frac{1}{t}\int \limits_{0}\limits^{t}\Xi(\X_{x}(\tau))d\tau- \nu\right|\right|_{W}\\
	& \leq & ~ 	\lim \limits_{t  \rightarrow \infty }\frac{1}{\sqrt{t}}\frac{1}{\sqrt{t}}\left|\left|\int \limits_{0}\limits^{t}\Xi_{\nu}(\X_{x}(\tau))d\tau- \sum \limits_{k=1}\limits^{L(t)+1}Y_{k}\right|\right|_{W}+\left|\left|\frac{1}{t} \sum \limits_{k=1}\limits^{L(t)+1}Y_{k}\right|\right|_{W}\\
	& = & ~ 0,
\end{eqnarray*}
with probability one, which proves (\ref{theorem_sllnclteq1}).\\
Finally, let us prove (\ref{theorem_sllnclteq2}). Consequently, from now on it is assumed that $(W,||\cdot||_{W})$ is a type $2$ Banach space. Let $(t_{n})_{n\in \mathbb{N}}\subseteq (0,\infty)$ be such that $\lim \limits_{n \rightarrow\infty}t_{n}=\infty$ and $(\theta_{n})_{n\in \mathbb{N}}$,  by $\theta_{n}:=\frac{t_{n}}{\mathbb{E}\alpha_{e_{\overline{x}}(1)}}$ for all $n \in \mathbb{N}$ and note that i) yields $\lim \limits_{n  \rightarrow \infty}\frac{L(t_{n})+1}{\theta_{n}}=1$ almost surely, and particularly in probability. Moreover, in light of ii), it is obvious that the sequence $(\frac{1}{\sqrt{\mathbb{E}\alpha_{e_{\overline{x}}(1)}}}Y_{n})_{n\in \mathbb{N}}$ is also centered, square integrable, i.i.d. and that each $\frac{1}{\sqrt{\mathbb{E}\alpha_{e_{\overline{x}}(1)}}}Y_{n}$ is distributed as $\frac{1}{\sqrt{\mathbb{E}\alpha_{e_{\overline{x}}(1)}}}Y_{0}$. These results enable us to employ Theorem \ref{theorem_anscombeclt}, which yields
\begin{align*}
\lim \limits_{n  \rightarrow \infty} \frac{1}{\sqrt{t_{n}}}\sum \limits_{k=1}\limits^{L(t_{n})+1}Y_{k}=\lim \limits_{n  \rightarrow \infty} \frac{1}{\sqrt{\theta_{n}}}\sum \limits_{k=1}\limits^{L(t_{n})+1}\frac{1}{\sqrt{\mathbb{E}\alpha_{e_{\overline{x}}(1)}}}Y_{k}=Z,
\end{align*}
in distribution. Finally, invoking iii) yields
\begin{align*}
\lim \limits_{n  \rightarrow \infty}\frac{1}{\sqrt{t_{n}}}\left(\int \limits_{0}\limits^{t_{n}}\Xi(\X_{x}(\tau))d\tau-t_{n}\nu\right)-\frac{1}{\sqrt{t_{n}}}\sum \limits_{k=1}\limits^{L(t_{n})+1}Y_{k}=0,
\end{align*}
almost surely and consequently
\begin{align*}
\lim \limits_{n  \rightarrow \infty}\frac{1}{\sqrt{t_{n}}}\left(\int \limits_{0}\limits^{t_{n}}\Xi(\X_{x}(\tau))d\tau-t_{n}\nu\right)= Z,
\end{align*}
in distribution, by \cite[Theorem 3.1]{Billingsley}, which gives the claim as $(t_{n})_{n\in \mathbb{N}}$ was arbitrary. (By the very definition of convergence in distribution it is clear that it suffices to consider sequences.)
\end{proof}

Now note that for $\Xi:V\rightarrow V_{2}$ with $\Xi(v):=v$, if $v \in V_{2}$ and $\Xi(v):=0$, if $v \in V \setminus V_{2}$, it is easy to verify that $(\Xi,(V_{2},||\cdot||_{V_{2}}))\in SL_{V_{2}}(V)$. Moreover, for $\xi:V\rightarrow \mathbb{R}$ with $\xi(v):=||v||_{V_{2}}$ if $v \in V_{2}$ and $\xi(v):=0$ for $v \in V \setminus V_{2}$, we also get $(\xi,\mathbb{R})\in SL_{V_{2}}(V)$. Using these facts together with the preceding theorem and Lemma \ref{lemma_Xbasicprop}.i) yields the following corollary.

\begin{corollary}\label{theorem_corvectorvalued} Let $x:\Omega \rightarrow V$ be an independent initial leading to extinction. Then the following assertions hold.
\begin{enumerate}
	\item $\lim \limits_{t\rightarrow \infty} \frac{1}{t}\int \limits_{0}\limits^{t}\X_{x}(\tau)d\tau= \nu_{1}$ almost surely in $(V_{2},||\cdot||_{V_{2}})$, where $\nu_{1}:=\frac{1}{\mathbb{E}(\alpha_{e_{\overline{x}}(1)})}\mathbb{E}\left(\int \limits_{0}\limits^{\alpha_{e_{\overline{x}}(1)}}\X_{\overline{x}}(\tau)d\tau\right)$.
	\item $\lim \limits_{t\rightarrow \infty} \frac{1}{t}\int \limits_{0}\limits^{t}||\X_{x}(\tau)||_{V_{2}}d\tau= \nu_{2}$ almost surely, where $\nu_{2}:=\frac{1}{\mathbb{E}(\alpha_{e_{\overline{x}}(1)})}\mathbb{E}\left(\int \limits_{0}\limits^{\alpha_{e_{\overline{x}}(1)}}||\X_{\overline{x}}(\tau)||_{V_{2}}d\tau\right)$.
	\item $\lim \limits_{t  \rightarrow \infty}\frac{1}{\sqrt{t}}\left(\int \limits_{0}\limits^{t}||\X_{x}(\tau)||_{V_{2}}d\tau-t\nu_{2}\right)=Z_{1}$ in distribution, where $Z_{1}\sim N(0,\sigma^{2})$ and $\sigma^{2} \in [0,\infty)$ is given by $\sigma^{2}:= \frac{1}{\mathbb{E}(\alpha_{e_{\overline{x}}(1)})}\mathbb{E}\left(\int \limits_{0}\limits^{\alpha_{e_{\overline{x}}(1)}}||\X_{\overline{x}}(\tau)||_{V_{2}}-\nu_{2} d\tau\right)^{2}$.
	\item If $(V_{2},||\cdot||_{V_{2}})$ is in addition of type $2$, then $\lim \limits_{t  \rightarrow \infty}\frac{1}{\sqrt{t}}\left(\int \limits_{0}\limits^{t}\X_{x}(\tau)d\tau-t\nu_{1}\right)=Z_{2}$ in distribution, as elements of $(V_{2},||\cdot||_{V_{2}})$, where $Z_{2}\sim N_{V_{2}}(0,Q)$ and $Q:=\Cov_{V_{2}}\left(\sqrt{\frac{1}{\mathbb{E}(\alpha_{e_{\overline{x}}(1)})}}\int \limits_{0}\limits^{\alpha_{e_{\overline{x}}(1)}}\X_{\overline{x}}(\tau)-\nu_{1} d\tau\right)$.
\end{enumerate}
\end{corollary}

\section{Asymptotic Results for the weighted p-Laplacian evolution Equation}
\label{plaplace}
The purpose of this section is to apply the results developed in Section \ref{section_asymptotics} to the weighted p-Laplacian evolution equation with Neumann boundary conditions on an $L^{1}$-space, for "small" values of $p$. The existence and uniqueness theory for this equation can be found in \cite{mazon}. Moreover, \cite{ich1} deals with asymptotic results for this equation.\\

Throughout this section, let $n \in \mathbb{N}\setminus \{1\}$  and $\emptyset \neq  S  \subseteq \mathbb{R}^{n}$ be a non-empty, open, connected and bounded sets of class $C^{1}$.
Moreover, let $p \in (1,\infty)\setminus \{2\}$ and set $L^{q}(S,\mathbb{R}^{m}):=L^{q}(S,\B(S),\lambda;\mathbb{R}^{m})$, for any $q \in [1,\infty]$ and $m \in \mathbb{N}$, where $\lambda$ denotes the Lebesgue measure. This is further abbreviated by $L^{q}(S)$, if $m=1$. In addition, introduce $L^{q}_{0}(S):=\{f\in L^{q}(S):~\overline{(f)}=0\}$, where $\overline{(f)}:=\frac{1}{\lambda(S)}\int \limits_{S}fd\lambda$.\\ 
Now, let $\gamma: S   \rightarrow (0,\infty)$ be such that $\gamma \in L^{\infty}( S )$, $\gamma^{\frac{1}{1-p}} \in L^{1}( S)$ and assume that there is an $A_{p}$-Muckenhoupt weight (see, \cite[page 4]{ich1}) $\gamma_{0}:\mathbb{R}^{n}\rightarrow \mathbb{R}$ such that $\gamma_{0}|_{ S }=\gamma$ a.e.  on $S$. Moreover, set
\begin{align*}
p_{0}:=\inf\{q>1:\gamma^{\frac{1}{1-q}}\in L^{1}(S)\}.
\end{align*}
It is plain that $p_{0}\leq p$. In fact, we even have $p_{0}<p$, cf. \cite[Lemma 4.3]{ich1}.
Moreover, $W_{\gamma}^{1,p}( S )$ denotes the weighted Sobolev space defined by
\begin{align*}
W_{\gamma }^{1,p}( S ):=\{f \in L^{p}( S ): f \text{ is weakly diff. and } ~\gamma^{\frac{1}{p}}\nabla f \in L^{p}( S;\mathbb{R}^{n})\}. 
\end{align*} 
Throughout this section, $|\cdot|_{n}$ is the Euclidean norm on $\mathbb{R}^{n}$ and for any $x,y\in\R^{n}$, $x\cdot y$ denotes the canonical inner product of these vectors.\\
Using these notations we introduce the following weighted p-Laplacian operator with Neumann boundary conditions:

\begin{definition}\label{definifition_plaplaceop} Let $A: D(A)\rightarrow 2^{L^{1}(S)}$ be defined by: $(f,\hat{f}) \in A$ if and only if the following assertions hold.
	\begin{enumerate}
		\item $f \in W^{1,p}_{\gamma}( S ) \cap L^{\infty}( S )$. 
		\item $\hat{f} \in L^{1}( S )$.
		\item $\int \limits_{ S}  \gamma|\nabla f|_{n}^{p-2}\nabla f\cdot\nabla \varphi  d \lambda = \int \limits_{ S } \hat{f} \varphi d \lambda$ for all $\varphi\in W^{1,p}_{\gamma }( S )\cap L^{\infty}( S )$.
	\end{enumerate}
\end{definition}

\begin{remark} It is an easy exercise to see that the integrals occurring in Definition \ref{definifition_plaplaceop}.iii) exist and are finite. Moreover, one also verifies that $A$ is single-valued, see \cite[Lemma 3.1]{ich1}.\\
In addition,  note that if one chooses $\gamma=1$ on $S$, then $A$ is simply the $p$-Laplacian operator with Neumann boundary conditions.
\end{remark}

\begin{remark} It turns out that $A$ is not m-accretive but that its closure is. Throughout this section, $\A:D(\A)\rightarrow 2^{L^{1}(S)}$ denotes the closure of $A$, i.e. $(f,\hat{f})\in \A$ if there is a sequence $((f_{m},\hat{f}_{m}))_{m \in \mathbb{N}}\subseteq A$ such that $\lim \limits_{m \rightarrow \infty} (f_{m},\hat{f}_{m})=(f,\hat{f})$, in $L^{1}(S)\times L^{1}(S)$.\\
Actually, it is possible to determine the closure explicitly, see \cite[Proposition 3.6]{mazon} or \cite[Definition 2.2]{ich1}. But the explicit description of the closure is quite technical and not needed for our purposes, therefore it will be omitted.
\end{remark}

\begin{remark}\label{lemma_a11} $\A$ is densely defined and m-accretive, see \cite[Theorem 3.7]{mazon}. In the sequel, let  $T_{\A}(\cdot)u:[0,\infty)\rightarrow L^{1}(S)$, where $u \in L^{1}(S)$, be such that $(T_{\A}(t))_{t\geq 0}$ is the semigroup associated to $\A$, see Remark \ref{remark_msex}.\\
Moreover, it is an easy exercise to deduce from  \cite[Lemma 3.3]{ich1} that $||T_{\A}(t)u||_{L^{q}(S)} \leq ||u||_{L^{q}(S)}$ for all $t\geq 0$, $u \in L^{q}(S)$ and $q\in [1,\infty]$. In addition, $T_{\A}$ preserves mass, i.e. $\overline{(T_{\A}(t)v)}=\overline{(v)}$ for all $v \in L^{1}(S)$, see \cite[Lemma 3.4]{ich1}. Combining these results yields that $(L^{q}_{0}(S),||\cdot||_{L^{q}(S)})$ as well as $(L^{q}(S),||\cdot||_{L^{q}(S)})$ are invariant with respect to $T_{\A}$, for all $q \in [1,\infty]$. Moreover, it is clear that the injections $L^{q}(S)\hookrightarrow L^{1}(S)$ and $L^{q}_{0}(S)\hookrightarrow L^{1}(S)$ are continuous. 
\end{remark}

The following lemma will be extracted from \cite{ich1} and enables us to apply the results of Section \ref{section_asymptotics} to the weighted $p$-Laplacian evolution equation.

\begin{lemma}\label{lemma_a12} Assume that the interval $\left(\frac{p_{0}(n-2)}{n+2}+p_{0},2\right)$ is non-empty and that $p \in \left(\frac{p_{0}(n-2)}{n+2}+p_{0},2\right)$. In addition, introduce $\rho:=2-p$ and
\begin{align*}
\kappa:=(2-p)\left(\tilde{C}_{ S }^{p} \left( C_{ S ,\frac{2n}{n+2}}^{\frac{2n}{n+2}}+1\right)^{\frac{np+2p}{2n}} \tilde{\Gamma}_{n,p} \right)^{-1}>0,
\end{align*}
where: $\tilde{C}_{S}$ is the operator norm of the continuous injection $W^{1,\frac{2n}{n+2}}\hookrightarrow L^{2}(S)$; $C_{S,\frac{2n}{n+2}}$ is the Poincar\'{e} constant (see \cite[p. 10]{ich1}) of $S$ in $L^{\frac{2n}{n+2}}(S)$; and
\begin{align*}
\tilde{\Gamma}_{n,p}:=\left( \int \limits_{ S } \gamma^{\frac{2n}{2n-np-2p}}d\lambda \right)^{\frac{np+2p-2n}{2n}}<\infty.
\end{align*}
Then we have 
\begin{align}
\label{lemma_assumptionextinctioneq1}
||T_{\A}(t)u||_{L^{2}(S)}^{\rho}\leq (-\kappa t +||u||_{L^{2}(S)}^{\rho})_{+},
\end{align}
for all $t\geq 0$ and $u\in L^{2}_{0}(S)$.
\end{lemma}
\begin{proof} Firstly, $\tilde{\Gamma}_{n,p}$ is indeed finite, see \cite[Lemma 5.4]{ich1}. Now assume $u \in D(A)\cap L^{2}_{0}(S)$ and introduce $t^{\ast}_{u}:= \inf(t\geq 0:~T_{\A}(t)u=0)$. If $t^{\ast}_{u}=0$, then (by continuity) $u=0$, and consequently $T_{\A}(t)u=0$ for all $t\geq 0$. In this case (\ref{lemma_assumptionextinctioneq1}) is trivial. Hence, assume $t^{\ast}_{u}>0$. Moreover, we have $t^{\ast}_{u}<\infty$, see \cite[Lemma 5.4]{ich1}. Now let $\varepsilon \in (0,t^{\ast}_{u})$ be arbitrary but fixed, introduce $f_{u}:[0,t_{u}^{\ast}-\varepsilon]\rightarrow [0,\infty)$, by $f_{u}(t):=||T_{\A}(t)u||_{L^{2}(S)}^{\rho}$ and $\hat{\varepsilon}:=||T_{\A}(t^{\ast}_{u}-\varepsilon)u||_{L^{2}(S)}^{2}>0$.\\
It can be inferred from the results in \cite{ich1} that $f_{u}$ is Lipschitz continuous, more precisely: The mapping $[0,t_{u}^{\ast}-\varepsilon]\ni t \mapsto ||T_{\A}(t)u||_{L^{2}(S)}^{2}$ is Lipschitz continuous, see \cite[Lemma 5.2]{ich1}. Moreover, it is common knowledge that $[\hat{\varepsilon},||u||_{L^{2}(s)}^{2}]\ni x \mapsto x^{\frac{\rho}{2}}$ is Lipschitz continuous as well, since $\rho \in (0,1)$ and by construction $\hat{\varepsilon}>0$. Consequently, $f_{u}$ is (as it is the composition of Lipschitz continuous functions) Lipschitz continuous.\\
Particularly, $f_{u}$ is differentiable almost everywhere and by \cite[Lemma 5.3]{ich1} we get
\begin{align*}
f_{u}^{\prime}(t)=\frac{\partial}{\partial t}\left(\int \limits_{S}T_{\A}(t)u^{2}\right)^{1-\frac{p}{2}}=(p-2)||T_{\A}(t)u||_{L^{2}(S)}^{-p}\int \limits_{S}\gamma|\nabla T_{\A}(t)u|_{n}^{p}d\lambda,
\end{align*}
for a.e. $t\in [0,t_{u}^{\ast}-\varepsilon]$. Consequently, appealing to \cite[Eq. (5.7)]{ich1}, yields $f_{u}^{\prime}(t)\leq -\kappa$ for a.e. $t\in [0,t_{u}^{\ast}-\varepsilon]$. Hence, we obtain
\begin{align*}
f_{u}(t)-f_{u}(0)= \int \limits_{0}\limits^{t}f_{u}^{\prime}(\tau)dt\leq -\kappa t,~\forall t \in [0,t_{u}^{\ast}-\varepsilon],
\end{align*}
i.e. $||T_{\A}(t)u||_{L^{2}(S)}^{\rho}\leq -\kappa t +||u||_{L^{2}(S)}^{\rho}$ which holds for all $t \in [0,t^{\ast}_{u})$, as $\varepsilon>0$ was arbitrary. Moreover, note that as $[0,\infty)\ni t \mapsto ||T_{\A}(t)u||_{L^{1}(S)}$ is a continuous, monotonically decreasing map, we have $T_{\A}(t)u=0$ for all $t \geq t^{\ast}_{u}$. Consequently, the preceding inequality enables us to conclude that
\begin{align*}
||T_{\A}(t)u||_{L^{2}(S)}^{\rho}\leq (-\kappa t +||u||_{L^{2}(S)}^{\rho})_{+},
\end{align*}
for all $t \geq 0$ and $u \in D(A)\cap L^{2}_{0}(S)$.\\ 
Now let $u \in L^{2}_{0}(S)$ be arbitrary. Then there is a sequence $(u_{m})_{m\in \mathbb{N}}\subseteq D(A)$, such that $\lim \limits_{m \rightarrow \infty}u_{m}=u$ in $L^{2}(S)$, cf. \cite[Lemma 5.6]{ich1}. Moreover, one instantly verifies that also $u_{m}-\overline{(u_{m})}\in D(A)$ for all $m \in \mathbb{N}$. Consequently, as $D(A)\subseteq L^{\infty}(S)\subseteq L^{2}(S)$, we get $u_{m}-\overline{(u_{m})}\in D(A)\cap L^{2}_{0}(S)$ for all $m \in \mathbb{N}$ and that
\begin{align*}
\lim \limits_{m \rightarrow \infty} u_{m}-\overline{(u_{m})}=u-\overline{(u)}=u,
\end{align*}
in $L^{2}(S)$. In addition, by continuity we have $\lim \limits_{m \rightarrow \infty} T_{\A}(t)(u_{m}-\overline{(u_{m})})=T_{\A}(t)u$ in $L^{1}(S)$ and (by passing to a subsequence if necessary) also almost everywhere on $S$. Conclusively, appealing to Fatou's Lemma yields
\begin{align*}
||T_{\A}(t)u||^{\rho}_{L^{2}(S)}\leq \liminf\limits_{m \rightarrow \infty} ||T_{\A}(t)(u_{m}-\overline{(u_{m})})||_{L^{2}(S)}^{\rho} \leq  (-\kappa t +||u||_{L^{2}(S)}^{\rho})_{+},
\end{align*}
for all $t\geq 0$.
\end{proof}

In the sequel, we assume that $\left(\frac{p_{0}(n-2)}{n+2}+p_{0},2\right)$ is non-empty, that $p \in \left(\frac{p_{0}(n-2)}{n+2}+p_{0},2\right)$ and that $\rho \in (0,1)$ and $\kappa \in (0,\infty)$ are as in the preceding lemma.\footnote{Note that if $n=2$ and $p_{0}=1$, then $\left(\frac{p_{0}(n-2)}{n+2}+p_{0},2\right)=(1,2)$ and that $p_{0}=1$ holds particularly if $\gamma$ is bounded from below away from zero. }\\ 
In addition, $(\eta_{m})_{m \in \mathbb{N}}$ and $(\beta_{m})_{m \in \mathbb{N}}$ denote i.i.d. sequences, where $\eta_{m}:\Omega \rightarrow L^{1}(S)$ and $\beta_{m}:\Omega \rightarrow (0,\infty)$ are $\F$-$\B(L^{1}(S))$-measurable and $\F$-$\B((0,\infty))$-measurable, respectively.  Moreover, assume that $(\eta_{m})_{m \in \mathbb{N}}$ and $(\beta_{m})_{m \in \mathbb{N}}$ are independent of each other.  As in the previous section, set $\alpha_{0}:=0$ and
$\alpha_{m}:=\sum\limits_{k =1}\limits^{m}\beta_{k}$ for all $m \in \mathbb{N}$. Moreover, let $x \in \mathcal{M}(\Omega;L^{1}(S))$ be jointly independent of $(\beta_{m})_{m\in \mathbb{N}}$ and $(\eta_{m})_{m\in \mathbb{N}}$; assume that $x \in L^{2}_{0}(S)$ a.s. and $||x||_{L^{2}(S)}^{2\rho}\in L^{1}(\Omega)$. Finally, let $\X_{x}:[0,\infty)\times \Omega \rightarrow L^{1}(S)$, be the process generated by $((\beta_{m})_{m \in \mathbb{N}},(\eta_{m})_{m \in \mathbb{N}},x,\A)$ in $L^{1}(S)$; and let $\overline{x}\in \mathcal{M}(\Omega;L^{1}(S))$ and $\alpha_{e_{\overline{x}}(1)}$ be as in Remark \ref{remark_xbar} and Definition \ref{definition_basic}.v), respectively.\\ 
Now assume that $\eta_{m}\in L^{2}_{0}(S)$ almost surely and that there is a constant $\hat{\varepsilon}>0$ such that\linebreak $\beta_{m}^{11+\hat{\varepsilon}},||\eta_{m}||_{L^{2}(S)}^{\rho(11+\hat{\varepsilon})} \in L^{1}(\Omega)$ and $-\kappa \mathbb{E}\beta_{m}+\mathbb{E}||\eta_{m}||_{L^{2}(S)}^{\rho}<0$.

\begin{theorem} Assume $||\eta_{m}||_{L^{2}(S)}\in L^{4}(\Omega)$ and introduce $\nu:=\frac{1}{\mathbb{E}(\alpha_{e_{\overline{x}}(1)})}\mathbb{E}\left(\int \limits_{0}\limits^{\alpha_{e_{\overline{x}}(1)}}\X_{\overline{x}}(\tau)d\tau\right)$. Then the convergence
	\begin{align*}
	\lim \limits_{t\rightarrow \infty} \frac{1}{t}\int \limits_{0}\limits^{t}\X_{x}(\tau)d\tau= \nu,
	\end{align*}
	takes place almost surely in $(L^{2}(S),||\cdot||_{L^{2}(S)})$. Moreover, we have
	\begin{align*}
	\lim \limits_{t  \rightarrow \infty}\frac{1}{\sqrt{t}}\left(\int \limits_{0}\limits^{t}\X_{x}(\tau)d\tau-t\nu\right)=Z,
	\end{align*}
	in distribution, as elements of  $(L^{2}(S),||\cdot||_{L^{2}(S)})$, where $Z \sim N_{L^{2}(S)}(0,Q)$ and the covariance is given by $Q:=\Cov_{L^{2}(S)}\left(\sqrt{\frac{1}{\mathbb{E}(\alpha_{e_{\overline{x}}(1)})}}\int \limits_{0}\limits^{\alpha_{e_{\overline{x}}(1)}}\X_{\overline{x}}(\tau)-\nu d\tau\right)$.
\end{theorem}
\begin{proof} This follows from Corollary \ref{theorem_corvectorvalued}, more precisely: Choose $V=L^{1}(S)$, $V_{1}=L^{2}_{0}(S)$ and $V_{2}=L^{2}(S)$, then combining Remark \ref{lemma_a11} and Lemma \ref{lemma_a12} yield Assumption \ref{assumption_fe}. Moreover, Assumption \ref{assumption_mb} holds by construction. Finally, it is well known that $L^{2}(S)$ is a type $2$ Banach space, see \cite[Theorem 3.4]{CLT}.
\end{proof}

\begin{theorem} Let $q \in [1,\infty)$ be given. Moreover, assume $x,\eta_{m} \in L^{q}(S)$ a.s. and $||\eta_{m}||_{L^{q}(S)}\in L^{4}(\Omega)$ and introduce $\nu:=\frac{1}{\mathbb{E}(\alpha_{e_{\overline{x}}(1)})}\mathbb{E}\left(\int \limits_{0}\limits^{\alpha_{e_{\overline{x}}(1)}}||\X_{\overline{x}}(\tau)||_{L^{q}(S)}d\tau\right)$. Then the convergence
	\begin{align*}
	\lim \limits_{t\rightarrow \infty} \frac{1}{t}\int \limits_{0}\limits^{t}||\X_{x}(\tau)||_{L^{q}(S)}d\tau= \nu,
	\end{align*}
	takes place almost surely. Moreover, 
	\begin{align*}
	\lim \limits_{t  \rightarrow \infty}\frac{1}{\sqrt{t}}\left(\int \limits_{0}\limits^{t}||\X_{x}(\tau)||_{L^{q}(S)}d\tau-t\nu\right)=Z
	\end{align*}
	in distribution, where $Z \sim N(0,\sigma^{2})$ and $\sigma^{2} \in [0,\infty)$ is given by
	\begin{align*}
	\sigma^{2}:= \frac{1}{\mathbb{E}(\alpha_{e_{\overline{x}}(1)})}\mathbb{E}\left(\int \limits_{0}\limits^{\alpha_{e_{\overline{x}}(1)}}||\X_{\overline{x}}(\tau)||_{L^{q}(S)}-\nu d\tau\right)^{2}.
	\end{align*}
\end{theorem}
\begin{proof} Analogously, all claims follow from Corollary \ref{theorem_corvectorvalued}, by choosing $V=L^{1}(S)$, $V_{1}=L^{2}_{0}(S)$ and $V_{2}=L^{q}(S)$.
\end{proof}

\begin{center}
	\textsc{Acknowledgment}
\end{center}
The present author is grateful to Prof. Dr. Alexei Kulik for fruitful conversations during a research stay of the present author at Technische Universit\"at Berlin.

%%%%%%%%%%%%%%%%%%%%%%%%%%%%%%%%
%%%%BEGIN bibliography%%%%%%%%%%%
%%%%%%%%%%%%%%%%%%%%%%%%%%%%%%%%


\begin{thebibliography}{9} 
\addcontentsline{toc}{chapter}{\ \quad Bibliography}

\bibitem{IDA}
{\sc C. D. Aliprantis, K. C. Border}: {\em Infinite Dimensional Analysis. Springer\/} (2005)

\bibitem{acmbook}
{\sc F. Andreu-Vaillo, V. Caselles, J.M. Maz\'{o}n}: {\em Parabolic Quasilinear Equations minimizing linear growth Functionals. Birkh\"auser\/} (2010)

\bibitem{mazon}
{\sc F. Andreu, J.M. Maz\'{o}n,  J. Rossi, J. Toledo}: {\em Local and nonlocal weighted p-Laplacian evolution Equations with Neumann Boundary Conditions. Publ. Math. 55, pp. 27-66\/} (2011)

\bibitem{BenilanBook}
{\sc P. B\'{e}nilan, M. Crandall, A. Pazy}: {\em Nonlinear Evolution Equations in Banach Spaces. Book to appear\/}; \url{http://www.math.tu-dresden.de/~chill/files/} 

\bibitem{Billingsley}
{\sc P. Billingsley}: {\em Convergence of Probability Measures. Wiley\/} (1999)

\bibitem{extinct}
{\sc J. Diaz}: {\em Special Finite Time Extinction in Nonlinear Evolution Systems: Dynamic Boundary Conditions and Coulomb Friction Type Problems. Progress in Nonlinear Differential Equations and Their Applications 64, pp. 71-97\/} (2005)

\bibitem{good}
{\sc A. Gut}: {\em Probability: A graduate Course. Springer\/} (2013)

\bibitem{CLT}
{\sc N. Jash}: {\em Central Limit Theorem in a Banach Space. Proceedings of the First International Conference on Probability in Banach Spaces, 20-26 July 1975, Oberwolfach\/} (1975)

\bibitem{katz}
{\sc M. Katz}: {\em The Probability in the Tail of a Distribution. The Annals of Mathematical Statistics\/} (1962) 

\bibitem{kechris}
{\sc Alexander Kechris}: {\em Classical Descriptive Set Theory. Springer\/} (1995) 

\bibitem{PIBS}
{\sc M. Ledoux, M. Talagrand}: {\em Probability in Banach Spaces. Springer\/} (1991) 

\bibitem{SIBS}
{\sc V. Mandrekar, B. R\" udiger}: {\em Stochastic Integration in Banach Spaces. Springer\/} (2015) 

\bibitem{tweety}
{\sc S.P. Meyn, R.L. Tweedie}: {\em Markov Chains and Stochastic Stability. Springer\/} (1993) 

\bibitem{ich1}
{\sc A. Nerlich}: {\em Asymptotic Results for Solutions of a weighted $p$-Laplacian evolution Equation with Neumann Boundary Conditions. Nonlinear Differential Equations and Applications\/} (2017)

\bibitem{ich2}
{\sc A. Nerlich}: {\em Abstract Cauchy Problems driven by random Measures: Existence and Uniqueness.\/} (submitted); Arxiv: \url{https://arxiv.org/pdf/1710.01796.pdf}

\bibitem{PR}
{\sc V. Paulauskas, A. Rackauskas}: {\em Approximation Theory in the Central Limit Theorem. Kluwer Academic Publishers\/} (1989) 









\end{thebibliography}
\end{document}